\documentclass[letterpaper,11pt]{amsart}
\usepackage{mathtools}
\usepackage{amsmath}
\usepackage{amssymb}
\usepackage{yhmath}
\usepackage{graphicx}
\usepackage{ mathrsfs }
\usepackage{bbm}
\usepackage{xcolor}
\usepackage{tikz-cd}
\usepackage{tikz}
\usetikzlibrary{patterns}
\usepackage{hyperref}

\DeclareMathAlphabet{\mathpzc}{OT1}{pzc}{m}{it}

\usepackage{thmtools}
\usepackage{thm-restate}

\usepackage{caption}

\newtheorem{theorem}{Theorem}[section]

\newtheorem*{claim*}{Claim}

\newtheorem{lemma}[theorem]{Lemma}
\newtheorem{lem}[theorem]{Lemma}
\newtheorem{corollary}[theorem]{Corollary}

\newtheorem{proposition}[theorem]{Proposition}

\newtheorem{prop}[theorem]{Proposition}%https://www.overleaf.com/project/616d911a46ea936b949fdb96

\theoremstyle{definition}
\newtheorem{definition}[theorem]{Definition}

\theoremstyle{remark}
\newtheorem{remark}[theorem]{Remark}

\numberwithin{equation}{section}

\newcommand{\op}{\operatorname}

\newcommand{\core}{\op{core}}
\newcommand{\hull}{\op{hull}}
\newcommand{\stab}{\op{stab}}

\newcommand{\C}{\mathbb C}
\newcommand{\R}{\mathbb R}
\newcommand{\Z}{\mathbb Z}
\newcommand{\N}{\mathbb N}
\renewcommand{\H}{\mathbb H}
\renewcommand{\S}{\mathbb S}
\newcommand{\M}{\mathsf M}

\newcommand{\horo}{\mathpzc{h}}
\newcommand{\Horo}{\mathpzc{H}}

\newcommand{\inte}{\op{int }}

\newcommand{\Isom}{\op{Isom}}
\newcommand{\PSL}{\op{PSL}}

\newcommand{\RF}{\mathrm{RF}}
\newcommand{\F}{\mathrm{F}}

\newcommand{\V}{\mathsf{V}}

\newcommand{\FM}{\mathrm{F}\mathsf{M}}
\newcommand{\RFM}{\mathrm{RF}\mathsf{M}}
\newcommand{\RFKM}{\mathrm{RF}_{k} \mathsf{M}}
\newcommand{\RFPM}{\mathrm{RF}_+\mathsf{M}}

\newcommand{\BFM}{\mathrm{BF}\mathsf{M}}
\newcommand{\ov}{\overline}
\newcommand{\cal}{\mathcal}
\newcommand{\Ga}{\Gamma}

\newcommand{\La}{\Lambda}

\newcommand{\ba}{\backslash}
\newcommand{\bb}{\mathbb}

\newcommand{\be}{\begin{equation}}
\newcommand{\ee}{\end{equation}}

\newcommand{\T}{\mathsf{T}}
\newcommand{\so}{\op{SO}^{\circ}}

\title[Horocycles in geometrically finite hyperbolic 3-manifolds]%{Horocycles in geometrically finite hyperbolic 3-manifolds with Fuchsian ends}
{Horocycles in hyperbolic 3-manifolds with round Sierpi\'nski limit sets}

\author{Dongryul M. Kim}

\address{
	Department of Mathematics, Yale University, New Haven, CT 06511, USA
}
\email{
	dongryul.kim@yale.edu
}

\author{Minju Lee}

\address{
	Department of Mathematics, University of Chicago, Chicago, IL 60637, USA
}
\email{
	minju1@uchicago.edu
}

\begin{document}
\begin{abstract}

Let $\M$ be a geometrically finite hyperbolic 3-manifold whose limit set is a round Sierpi\'nski gasket, i.e. $\M$ is geometrically finite and acylindrical with a compact, totally geodesic convex core boundary. In this paper, we classify orbit closures of the 1-dimensional horocycle flow on the frame bundle of $\M$. As a result, the closure of a horocycle in $\M$ is a properly immersed submanifold. This extends the work of McMullen-Mohammadi-Oh where $\M$ is further assumed to be convex cocompact.

\end{abstract}

\maketitle

%\tableofcontents

%%%%%%%%%%%%%%%%%%%%%%%%%%%%%%%%%%%%%%%%%%%%%%%%%%%%%%%%%%%%%%%%%%%%%%
%
%	Notes
%
%%%%%%%%%%%%%%%%%%%%%%%%%%%%%%%%%%%%%%%%%%%%%%%%%%%%%%%%%%%%%%%%%%%%%%

\section{Introduction}

Let $\M$ be a complete hyperbolic 3-manifold, and let $\chi \subset \M$ be an isometrically immersed copy of \(\mathbb{R}\) with torsion zero and geodesic curvature 1, referred to as a \emph{1-dimensional horocycle} or simply a \emph{horocycle}. Shah \cite{Shah1992master} and Ratner \cite{ratner} classified the closure $\ov{\chi} \subset \M$ in the case $\op{Vol}(\M) < \infty$, proving that $\ov{\chi}$ is a properly immersed submanifold of $\M$.
This classification was generalized to infinite-volume hyperbolic 3-manifolds by McMullen-Mohammadi-Oh in \cite{McMullen2016horocycles}, where they considered convex cocompact hyperbolic 3-manifolds with round Sierpi\'nski limit sets.

We call $\M$ \emph{convex cocompact} if its convex core $\core(\M)$ is compact, and \emph{geometrically finite} if the unit neighborhood of $\core(\M)$ has finite volume.
We say that $\M$ has a \emph{round Sierpi\'nski limit set} if the limit set $\La \subset \widehat \C$ of the Kleinian group $\pi_1(\M) <  \PSL_2(\C)$ is a round Sierpi\'nski gasket, i.e., 
$$\widehat \C - \La = \bigcup_{i = 1}^{\infty} B_i$$
is a countable union of round open disks $B_i \subset \widehat \C$ with disjoint closures (see Figure \ref{fig.gasket}).

A geometrically finite hyperbolic 3-manifold $\M$ has a round Sierpi\'nski limit set if and only if $\core(\M)$ has a non-empty interior and a compact, totally geodesic boundary.
Moreover, such $\M$ is acylindrical\footnote{A 3-manifold is called acylindrical if its compact core (also called Scott core) has incompressible boundary and every essential cylinder therein is boundary-parallel.} and has no rank-1 cusps (Lemma \ref{lem.onlyrank2}). Indeed, as shown by Thurston \cite{Thurston1986hyperbolic} and McMullen \cite[Corollary 4.3]{McMullen_iteration}, every geometrically finite, acylindrical hyperbolic 3-manifold $\M$ with compact $\partial \core(\M)$ is quasiconformally conjugate to a unique one with a round Sierpi\'nski limit set.

\begin{figure}[h]
\vspace{-3.5em}
\includegraphics[scale=0.35]{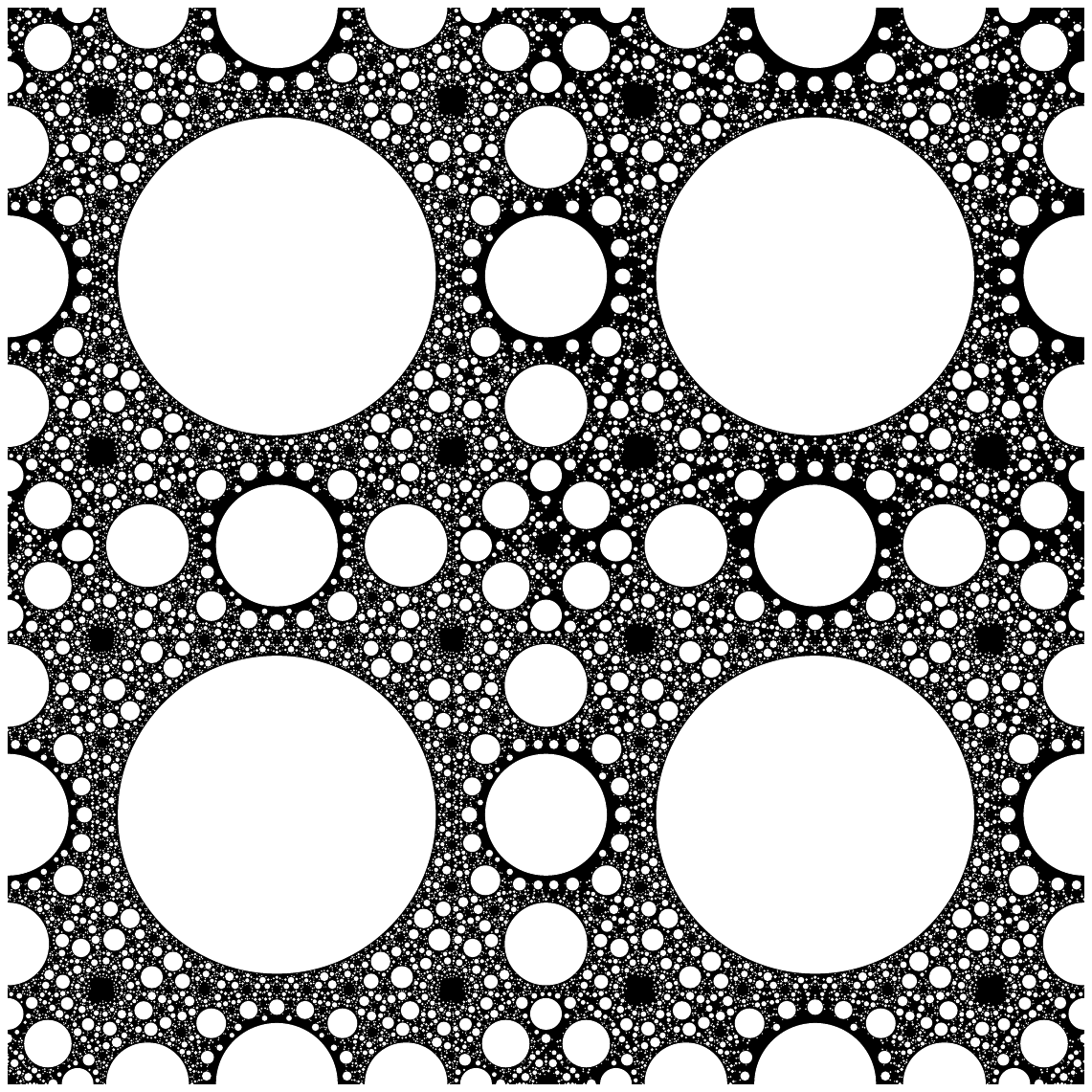}
\vspace{-3.5em}
        \caption[A round Sierpi\'nski limit set on $\widehat \C - \{\infty\}$ ]{A round Sierpi\'nski limit set with rank-2 parabolic limit points, drawn  on $\widehat \C - \{\infty\}$. The point $\infty$ is also a parabolic limit point. \footnotemark} \label{fig.gasket}
\end{figure}
 \footnotetext{Image credit: Yongquan Zhang}

Convex cocompact hyperbolic 3-manifolds with round Sierpi\'nski limit sets have been the only known infinite-volume examples where the topological behavior of closures of horocycles is fully understood \cite{McMullen2016horocycles}. In this paper, we extend the classification to geometrically finite 3-manifolds:

\begin{theorem} \label{thm.mainmnfd}
    Let $\M$ be a geometrically finite hyperbolic 3-manifold with a round Sierpi\'nski limit set. For any 1-dimensional horocycle $\chi \subset \M$, one of the following holds:
    \begin{enumerate}
        \item $\chi = \ov{\chi}$ is closed.
        \item $\ov{\chi}$ is a 2-dimensional compact horosphere.
        \item $\ov{\chi} $ is a properly immersed 2-manifold, parallel to a totally geodesic surface $\mathsf{S} \subset \M$.
        \item $\ov{\chi}$ is the entire 3-manifold $\M$.
    \end{enumerate}
\end{theorem}

\subsection*{Horocycle flows on frame bundles} As in \cite{McMullen2016horocycles}, Theorem \ref{thm.mainmnfd} is a consequence of the classification of orbit closures of the horocycle flow on the frame bundle of $\M$. To be precise, let $G = \PSL_2 (\C) = \Isom^+(\H^3)$ and consider the following subgroups: 
 $$
 \begin{aligned}
 H & = \PSL_2(\R), &
     N & = \left\{ n_z = \begin{pmatrix} 1 & z \\ 0 & 1 \end{pmatrix} : z \in \C \right \}, \\
 U &= \{n_s : s \in \R \}, \text{ and} &  V & = \{n_{is} : s \in  \R \}.
 \end{aligned}
 $$

For a hyperbolic 3-manifold $\M = \Ga \ba \H^3$ with an associated Kleinian group $\Ga < G$, we have the identification of its frame bundle $\FM$ with $\Ga \ba G$:
$$
\FM = \Ga \ba G.
$$
Then every (oriented) horocycle $\chi \subset \M$ lifts uniquely to a $U$-orbit $xU \subset \FM$ for some $x \in \FM$ and vice versa. 

We denote by $$\RFPM \subset \FM$$
the set of all frames directed toward $\core(\M)$ under the frame flow;  that is, their forward trajectories project to geodesic rays in $\M$ that remain within a bounded distance of $\core(\M)$. For a precise definition, see \eqref{eqn.rfpm}. This set $\RFPM$ is $U$-invariant, and any $U$-orbit outside $\RFPM$ is a properly immersed copy of $\R$. Hence, interesting dynamics appear only within $\RFPM$.
We now state our classification of $U$-orbit closures in $\FM$.

\begin{theorem} \label{thm.main}
    Let $\M$ be a geometrically finite hyperbolic 3-manifold with a round Sierpi\'nski limit set. Then for any $x \in \FM$, one of the following holds:
    \begin{enumerate}
        \item $xU$ is closed.
        \item $\ov{xU} = xN$ which is compact.
        \item $\ov{xU} = xvHv^{-1} \cap \RFPM$ for some $v \in V$.
        \item $\ov{xU} = \RFPM$.
    \end{enumerate}
\end{theorem}

Note that $N$-orbits and $H$-orbits in $\FM$ project to 2-dimensional horospheres in $\M$ and the images of totally geodesic immersions of a hyperbolic plane into $\M$, respectively. In particular, a compact $N$-orbit in $\FM$ corresponds to a compact horosphere in $\M$, and a closed $H$-orbit in $\FM$ corresponds to a totally geodesic plane in $\M$. Therefore, Theorem \ref{thm.mainmnfd} follows from Theorem \ref{thm.main}.

When $\M$ is convex cocompact, Thoerem \ref{thm.main} was proved by McMullen-Mohammadi-Oh \cite{McMullen2016horocycles} and (2) does not occur in that case. This was extended by Lee-Oh \cite{LO_orbit} to higher-dimensional convex cocompact hyperbolic manifolds whose convex cores have non-empty interiors and totally geodesic boundaries. In their work, they classified orbit closures of any connected, closed subgroup of
$\so(n, 1) = \Isom^+(\H^n)$ generated by unipotent elements.

\begin{remark}
    We emphasize that $U < G$ is a \emph{non-maximal} unipotent subgroup, i.e. $U$ is not a horospherical subgroup of $G$. This non-maximality introduces a fundamental difficulty in studying behavior of $U$-orbits. For these reasons, manifolds with round Sierpi\'nski limit sets, or equivalently, convex cores with compact, totally geodesic boundaries, were considered in \cite{McMullen2016horocycles} and \cite{LO_orbit} as well as in this paper.

    Indeed, the orbit closure of a horospherical subgroup has been classified for \emph{any} geometrically finite hyperbolic manifold by Dal'bo \cite{Dalbo2000topologie}, Ferte \cite{ferte_horosphere}, and Winter \cite{Winter2015mixing}, without requiring any additional geometric assumption.

\end{remark}

\subsection*{On the proof}
We briefly outline our strategy. We adapt the idea of McMullen-Mohammadi-Oh \cite{McMullen2016horocycles}, which uses the classification of closures of geodesic planes to classify closures of horocycles, to our setting. On the other hand, the presence of cusps poses an obstacle to directly applying the arguments in \cite{McMullen2016horocycles}. To address this, we incorporate techniques developed by Dani-Margulis \cite{Dani1989} and Shah \cite{Shah1992master} as well as the specific features of a limit set,
which ensure that every compact horocycle is contained within a compact horosphere.

Given a horocycle $\chi \subset \M$, we assume that $\chi$ is neither closed nor is its closure $\ov{\chi}$ a compact horosphere.
We then show that its closure $\ov{\chi}$ is a surface parallel to a closed geodesic plane or equal to $\M$. Our proof proceeds in two major steps:

The first step is to prove that, under the given hypothesis, $\ov{\chi}$ contains a surface in $\M$ equidistant from a geodesic plane (Proposition \ref{prop.main1}). 
This step relies on the classification of closures of geodesic planes by Benoist-Oh \cite{Benoist2022geodesic}. To achieve this, we consider the following two cases: 
\begin{itemize}
    \item[(a)]  $\ov{\chi}$ contains a compact horocycle $\chi_0 \subset \M$;
    \item[(b)]  $\ov{\chi}$ does not contain any compact horocycle in $\M$.
\end{itemize}

A key observation in handling case (a) is that every compact horocycle is contained in a compact horosphere due to the specific feature of a limit set.  Using the unipotent blowup developed in (\cite{Dani1989}, \cite{Shah1992master}), we scatter $\chi_0$ along geodesics or horospheres and deduce that $\overline{\chi}$ contains either a surface or a compact horosphere (Corollary \ref{cor.clu1}). Then employing the expansion of horospheres, we obtain that $\ov{\chi}$ always contains a surface (Theorem \ref{thm.clu2}). We remark that this is the place where the presence of cusp introduces an obstruction to directly adapting the arguments in \cite{McMullen2016horocycles}.

To address case (b), we utilize the notion of relatively $U$-minimal sets introduced in \cite{McMullen2017geodesic}. As shown in \cite{Benoist2022geodesic}, horocycles in $\M$ exhibit recurrence to certain compact subsets. Based on the recurrence, we adapt the approach of \cite{McMullen2016horocycles} to our setting and show that $\ov{\chi}$ is scattered along a geodesic ray or a horosphere (Lemma \ref{lem.newdynamics}). From this, we deduce that $\ov{\chi}$ contains a surface (Theorem \ref{thm.xu3}).

As the second step, we show that if $\ov{\chi}$ contains a surface in $\M$, then either $\ov{\chi}$ is entirely contained within the surface, or the surface is scattered along a horosphere within $\ov{\chi}$, from which the conclusion follows. In this step, we use the recurrence properties of horocycles established in \cite{Benoist2022geodesic} and 
the unipotent blowup for such horocycles (\cite{McMullen2017geodesic}, \cite{McMullen2016horocycles}).

\subsection*{Geodesic planes}
As mentioned above, we heavily use Besnoist-Oh's classification of closures of geodesic planes \cite{Benoist2022geodesic} to classify closures of horocycles.
For other works classifying closures of geodesic planes in infinite-volume hyperbolic manifolds, see (\cite{McMullen2017geodesic}, \cite{LO_orbit}, \cite{McMullen2022geodesic}, \cite{zhang_existence}, \cite{Benoist2022geodesic}, \cite{Khalil_planes}, \cite{TZ_geodesic}, \cite{jeoung2023operatorname}).

\subsection*{Open question}
It is an open question whether a similar classification of closures of horocycles holds for a general geometrically finite (acylindrical) hyperbolic 3-manifold.

\subsection*{Structure of the paper}
\begin{itemize}
    \item In Section \ref{sec.background} and Section \ref{sec.3mnfd}, we fix notations and terminologies used throughout the paper.
    \begin{itemize}
        \item Section \ref{sec.background} is about subgroups of $\PSL_2(\C)$, and
        \item Section \ref{sec.3mnfd} is about Kleinian groups and geometrically finite 3-manifolds. 
    \end{itemize} 
    \item Section \ref{sec.fn} is devoted to the recurrence of horocycle flows and the unipotent blowup lemmas. 
    \item In Section \ref{sec.exphoro}, we explain the expansion of horospheres and deduce properties of horocycles contained in an expanding sequence of horospheres by applying unipotent blowup. 
    \item Closed geodesic planes and the horocycles contained within them are discussed in Section \ref{sec.geodesicplane}.

    \item In Section \ref{sec.xwh}, we prove Theorem \ref{thm.main} for the closure of a horocycle that contains a closed geodesic plane. 

    \item In Section \ref{sec.xwn}, we classify closures of horocycles intersecting compact horospheres.

    \item The classification of closures of horocycles that do not contain any compact horocycle is addressed in Section \ref{sec.xwou}.

    \item Finally, we prove Theorem \ref{thm.main} in Section \ref{sec.fin}.

\end{itemize}

\subsection*{Acknowledgements}
We thank our advisor, Professor Hee Oh, for introducing us to this topic, suggesting this problem, and providing invaluable inspiration, insightful discussions, and many helpful comments on earlier drafts of this paper.
We also thank Yongquan Zhang for providing us with a beautiful image (Figure \ref{fig.gasket}). Finally, we thank the referee for their helpful comments and their careful reading of the original manuscript.

%%%%%%%%%%%%%%%%%%%%%%%%%%%%%%%%%%%%%%%%%%%%
%%%%%%%%%%%%%%%%%%%%%%%%%%%%%%%%%%%%%%%%%%%%
%%%%%%%%%%%%%%%%%%%%%%%%%%%%%%%%%%%%%%%%%%%%
%%%%%%%%%%
%%%%%%%%%% Backgrounds(?)
%%%%%%%%%%
%%%%%%%%%%%%%%%%%%%%%%%%%%%%%%%%%%%%%%%%%%%%
%%%%%%%%%%%%%%%%%%%%%%%%%%%%%%%%%%%%%%%%%%%%
%%%%%%%%%%%%%%%%%%%%%%%%%%%%%%%%%%%%%%%%%%%%

\section{Subgroups of $\PSL_2(\C)$}\label{sec.background}

In this section, we introduce basic notions and fix notations for subgroups of $\PSL_2(\C)$ that we use throughout the paper.
We mainly use the upper half-space model for the hyperbolic 3-space $\H^3 = \{ (z, t) \in \C \times \R : t > 0 \}$ whose boundary is the Riemann sphere $\widehat \C = \C \cup \{\infty\}$. We fix a basepoint $o = (0, 1) \in \H^3$. The group of orientation-preserving isometries on $\H^3$ is identified with $\PSL_2(\C)$, and its action on $\H^3$ extends to a conformal action of $\PSL_2(\C)$ on $\widehat \C$ given by linear fractional transformations. Including the ones in the introduction, we fix the following notations for subgroups of $\PSL_2(\C)$:
\be \label{eqn.notation}
\begin{aligned} 
    G & := \PSL_2(\C) \\
    K & := \stab_G(o) \cong \op{PSU}(2)\\
  H & := \left\{ \begin{pmatrix} a & b \\ c & d \end{pmatrix} \in G : a, b, c, d \in \R \right\} \cong \PSL_2(\R) \\
  A & : = \left\lbrace a_t := \begin{pmatrix}
    e^{t/2} & 0 \\
    0 & e^{-t/2}
  \end{pmatrix} : t \in \R \right\rbrace\\
  M & := \{ a_{i\theta} = \begin{pmatrix}
    e^{i \theta /2} & 0 \\
    0 & e^{-i \theta /2}
  \end{pmatrix} : \theta \in \R \} \cong \op{PSU}(1) \cong \S^1 \\
  N & := \left\lbrace n_z := \begin{pmatrix}
    1 & z \\
    0 & 1
  \end{pmatrix} : z \in \C \right\rbrace\\
  U & := \left\{u_s := 
  \begin{pmatrix}
    1 & s \\
    0 & 1
  \end{pmatrix}
  : s \in \R \right\} < N \\
  V & := \left\{v_s := 
  \begin{pmatrix}
    1 & is  \\
    0 & 1
  \end{pmatrix}  
   : s \in \R \right\} < N
\end{aligned}\ee
We then have the identification
$$
G/K = \H^3 \quad \text{and} \quad G/MAN = \widehat \C
$$
where identity cosets $K \in G/K$ and $MAN \in G/MAN$ correspond to $o \in \H^3$ and $\infty \in \widehat \C$ respectively. Moreover, equipping $G/K$ and $G/MAN$ with the left-multiplications by $G$, the above identifications are $G$-equivariant.

Let $\widehat \R = \R \cup \{\infty\} \subset \widehat \C$ be the standard circle given by the real axis. Then $H$ is the orientation-preserving stabilizer of $\widehat \R$ in $G$. In addition, $\widehat \R$ is the boundary of the $\H^2$-copy in $\H^3$ invariant under $H$. Observes also that $$AU \subset H.$$

Note that $A$ and $M$ commute, and $U$ and $V$ commute.
For a subgroup $S < G$, we denote by $\op{N}_G(S) < G$ the subgroup consisting of the normalizers of $S$ in $G$. Then
$$AM \subset \op{N}_G(N), \quad  AV \subset \op{N}_G(U), \quad \text{and} \quad AU \subset \op{N}_G(V).$$
Moreover, for $n_z \in N$, we have
$$a_t^{-1} n_z a_t = n_{z e^{-t}}\to e \quad \text{as } t \to + \infty.$$

Throughout the paper, we use the notations $a_t$, $n_z$, $u_s$, $u_t$, $v_s$, and  $v_t$ to represent matrices as defined in \eqref{eqn.notation}. Abusing notations, we occasionally use $a_n$, $u_n$, or $v_n$ to represent sequences in $A$, $U$, or $V$ respectively, where $n$ serves as an index rather than a matrix value. When sequences explicitly track the values of matrices, we use sequences in $\R$ and notations in \eqref{eqn.notation}. For example, we take a sequence $t_n \in \R$ and consider $a_{t_n} = \begin{pmatrix} e^{t_n/2} & 0 \\ 0 & e^{-t_n/2} \end{pmatrix} \in A$.

\subsection*{Frame bundle}
Denote by $\F \H^3$ the frame bundle over $\H^3$, the space of all (positively oriented, orthonormal) frames on $\H^3$.
The induced action of $G$ on $\F \H^3$ is transitive, and hence we identify
$$G = \F \H^3$$
so that the quotient map $G \to G/K$ becomes the basepoint projection $\F \H^3 \to \H^3$, and the right-multiplication by $A$ and $U$ give the (geodesic) frame flow and horocycle flow on $\F \H^3$, in directions of the first and the second components of a frame respectively.

For $g \in G$, let
$$g^+ := g(\infty) \in \widehat \C \quad \text{and} \quad g^- := g(0) \in \widehat \C.$$
Via the map $G \to G/K$, the orbit $gA$ projects to the bi-infinite geodesic $gA o \subset \H^3$ with endpoints $g^{\pm} \in \widehat \C$. Similarly, the orbit $g \cdot \{a_t \in A : t \ge 0\}$ projects to the geodesic ray in $\H^3$ based at $g  o \in \H^3$ and toward $g^+ \in \widehat \C$. Moreover, noting that $N < G$ projects to the horizontal plane $N o = \{ (z, 1) \in \H^3 : z \in \C \}$, the orbit $gN  o \subset \H^3$ is the horosphere passing through $g o \in \H^3$ and resting at $g^+ \in \widehat \C$. For each $n \in N$, the frame $gn$ is based at the horosphere $gN o$ and its first component is toward $(gn)^+ = g^+ \in \widehat \C$.

Finally, $H o \subset \H^3$ is the oriented copy of $\H^2$ whose boundary is $\widehat \R \subset \widehat \C$ and $H$ corresponds to the set of all frames whose first two components are restricted to positively oriented frames on the geodesic plane $H o$. Hence, $g H o \subset \H^3$ is the geodesic plane spanned by the first two components of the frame $g \in G$ and $gH$ is the set of all frames whose first two components are restricted to positively oriented frames on $gH o$. The boundary of the geodesic plane $gHo$ is the circle $ g  \widehat \R = \{ gh^+ \in \widehat \C : h \in H \}$.

%%%%%%%%%%%%%%%%%%%%%%%%%%%%%%%%%%%%%%%%%%%%
%%%%%%%%%%%%%%%%%%%%%%%%%%%%%%%%%%%%%%%%%%%%
%%%%%%%%%%%%%%%%%%%%%%%%%%%%%%%%%%%%%%%%%%%%
%%%%%%%%%%
%%%%%%%%%% Hyperbolic 3-mnfds
%%%%%%%%%%
%%%%%%%%%%%%%%%%%%%%%%%%%%%%%%%%%%%%%%%%%%%%
%%%%%%%%%%%%%%%%%%%%%%%%%%%%%%%%%%%%%%%%%%%%
%%%%%%%%%%%%%%%%%%%%%%%%%%%%%%%%%%%%%%%%%%%%

\section{Hyperbolic 3-manifolds} \label{sec.3mnfd}

A discrete subgroup $\Ga < G$ is called a Kleinian group, and the quotient $\M := \Ga \ba \H^3$ is a complete hyperbolic orbifold (manifold if $\Ga$ is torsion-free). We denote by $\La \subset \widehat \C$ the limit set of $\Ga$, which is defined as the set of accumulation points of $\Ga o \subset \H^3$ in the compactification $\H^3 \cup \widehat \C$. When $\# \La \ge 3$, $\Ga$ is called non-elementary. In this case, the $\Ga$-action on $\La$ is minimal and $\# \La = \infty$.

The convex core of $\M$ is defined as
$$\core (\M) := \Ga \ba \hull(\La) \subset \M$$
where $\hull(\La) \subset \H^3$ is the convex hull of $\La$ in $\H^3$.
We call $\Ga$ and $\M$ \emph{geometrically finite } if the unit neighborhood of $\core(\M)$ has finite volume. 

In this paper, we are interested in a geometrically finite Kleinian group $\Ga$ such that $\La$ is the round Sierpi\'nski gasket. In other words,
$$\widehat \C - \La = \bigcup_{i = 1}^{\infty} B_i$$ is a countable union of round open disks $B_i \subset \widehat \C$ with disjoint closures (Figure \ref{fig.gasket}). It is clear from the definition that $\Ga$ is non-elementary, and moreover Zariski dense in $G \cong \so(3, 1)$.
We say that such $\Ga$ and $\M = \Ga \ba \H^3$ have \emph{round Sierpi\'nski limit set}.

We remark that these conditions imply that $\M$ is acylindrical in the sense of \cite{Thurston1986hyperbolic}, and $\partial \core(\M)$ is compact and totally geodesic. Indeed, every geometrically finite, acylindrical hyperbolic 3-manifold $\M$ with $\partial \core(\M)$ compact is a quasiconformal deformation of a unique geometrically finite hyperbolic 3-manifold with a round Sierpi\'nski limit set (\cite{Thurston1986hyperbolic}, \cite[Corollary 4.3]{McMullen_iteration}).

\subsection*{Conical and parabolic limit points}
Let $\Ga < G$ be a Kleinian group with the limit set $\La \subset \widehat \C$. A limit point $x \in \La$ is called \emph{conical}  if any geodesic ray in $\H^3$ toward $x$ has an accumulation point in the quotient $\M = \Ga \ba \H^3$ and  \emph{parabolic} if $x$ is fixed by a parabolic element of $\Ga$, an element conjugate to $\begin{pmatrix} 1 & 1 \\ 0 & 1 \end{pmatrix}$. For a parabolic limit point $x \in \La$, its stabilizer $\stab_{\Ga}(x)$ in $\Ga$ is virtually abelian, and its rank is called \emph{rank} of $x$ and is either 1 or 2.
A parabolic limit point $x \in \La$ is called \emph{bounded parabolic} if the $\stab_{\Ga}(x)$-action on $\La - \{x\}$ is cocompact.

When $\Ga$ is geometrically finite, the limit set $\La$ is a disjoint union of conical limit points and bounded parabolic limit points. Moreover, there are finitely many bounded parabolic limit points $x_1, \cdots, x_n \in \La$ so that
$$
\La = \{ \text{conical limit points} \} \cup \bigcup_{i = 1}^n \Ga x_i.
$$
In particular, there are at most countably many parabolic limit points.
 
\begin{lemma}[{cf. \cite[Lemma 11.2]{Benoist2022geodesic}}] \label{lem.onlyrank2}
    Let $\Ga < G$ be geometrically finite with a round Sierpi\'nski limit set $\La$. Then every parabolic limit point is of rank 2.
\end{lemma}

\begin{proof}
    Suppose that there exists a parabolic limit point of rank 1, say $\infty \in \La$ without loss of generality.
    Since $\Ga$ is geometrically finite, $\infty$ is bounded parabolic, and hence $\stab_{\Ga}(\infty)$ acts cocompactly on $\La - \{ \infty\}$. Since $\stab_{\Ga}(\infty)$ is virtually conjugate to the subgroup $\begin{pmatrix} 1 & \Z \\ 0 & 1 \end{pmatrix}$, there are two parallel lines $L_1, L_2 \subset \C$ such that $\La - \{\infty\}$ is contained in the region bounded by $L_1$ and $L_2$. Then there exist two components $B_1, B_2 \subset \widehat \C - \La$ such that $L_1 \subset B_1$ and $L_2 \subset B_2$. Since $\La - \{ \infty \}$ is bounded by $L_1$ and $L_2$, $B_1 \neq B_2$. On the other hand, since $L_1$ and $L_2$ are lines in $\C$, their closures in the Riemann sphere $\widehat \C$ are circles passing through $\infty \in \widehat \C$, and hence $\ov{B_1} \cap \ov{B_2} \neq \emptyset$. This contradicts the hypothesis that $\La$ is a round Sierpi\'nski limit set.
\end{proof}

\subsection*{Renormalized frame bundle}

Let $\Ga < G$ be a Kleinian group and $\M := \Ga \ba G$. Since the identification $\F \H^3 = G$ is $G$-equivariant, this induces the identification of the frame bundle $\FM$ of $\M$ with $\Ga \ba G$:
$$\FM = \Ga \ba G.$$
Then frame flow and horocycle flow on $\F \H^3$ descend to $\FM$ and they are given as the right-multiplications of $A$ and $U$ on $\Ga \ba G$ respectively. We denote by the projection 
$$
\pi : \Ga \ba G \to \Ga \ba G / K
$$ which is the basepoint projection $\FM \to \M$. Throughout the paper, we denote by $[g] \in \Ga \ba G$ the coset $\Ga g$ for $g \in G$, and we refer to elements of $\Ga \ba G$ and $G$ as frames in $\M$ and $\H^3$, respectively.

Interesting dynamics arise in certain subsets of $\FM$. We define the renormalized frame bundle over $\M$ as  $$\RF \M := \{ [g] \in \Gamma \backslash G : g^{\pm} \in \Lambda \} \subset \F \M.$$ 
This is  the closed set consisting of all frames in $\M$ whose orbits under the frame flow are based at $\core (\mathsf{M})$. It is clear that $\RF \M$ is $AM$-invariant. We also set 
\be \label{eqn.rfpm}
\RF_+ \M := \RF \M \cdot N = \{[g] \in \Gamma \backslash G : g^+ \in \Lambda \}
\ee which is $MAN$-invariant. The projection
$\pi |_{\RFPM} : \RF_+ \M \to \M$ is surjective.

\subsection*{Boundary frames}

In the rest of the section, let $\Ga < G$ be a geometrically finite Kleinian group with a round Sierpi\'nski limit set, and $\M := \Ga \ba G$. 
Since $\partial \core(\M)$ is totally geodesic and compact, there are finitely many elements $z_1, \cdots, z_n \in \Ga \ba G$ with compact $H$-orbits
so that the set of frames whose first two components are tangent to $\partial \core(\M)$ is equal to  $\bigcup_{i = 1}^{n} z_i H \subset \Ga \ba G$. We set
\be \label{eqn.bfmdef}
\BFM := \bigcup_{i = 1}^n z_i H
\ee
and call boundary frames. Note that $\BFM \subset \RFM$ and $\BFM \cdot H = \BFM$.

Due to the specific feature of the limit set $\La$, we observe the following:

\begin{lemma}\label{lem.zv}
    For any $x\in\RFPM$, 
    $$x \in \RFM \cdot U \quad \text{or} \quad x \in \BFM \cdot V.$$
\end{lemma}
\begin{proof}
    Let $x = [g] \in \RFPM - \RFM \cdot U$ for $g \in G$. Then for any $u \in U$, $(gu)^- \notin \La$ since $(gu)^+ = g^+ \in \La$. Then the set $\{ (gu)^- : u \in U \}$ is a connected subset of $\widehat{\C} - \La$, and hence for some component $\Omega_0$ of $\widehat \C - \La$, $\{ (gu)^- : u \in U \} \subset \Omega_0$ and $g^+ \in \partial \Omega_0$. Then $\{ g^+ \} \cup \{ (gu)^- : u \in U \}$ is the round circle in $\overline{\Omega_0}$ tangent to $\partial \Omega_0$ at $g^+$. This implies that for some $v \in V$, $\{g^+ \} \cup \{ (gvu)^- : u \in U \} = \partial \Omega_0$. Noting that $(gv)^+ = g^+$, it follows that $xv = [gv] \in \BFM$.
\end{proof}

\begin{proposition} \cite[Theorem 4.1]{McMullen2016horocycles} \label{prop.intorfm}
    Let $x_n \in \RFM \cdot U$ be a sequence such that $x_n \to y \in \RFM$ as $n \to \infty$.
  \begin{enumerate}
      \item if $y\not\in\BFM$, then there exists a sequence $u_n \to e$ in $U$  such that $x_n u_n \in \RF \M$ for all $n \ge 1$.
  In particular, $x_nu_n\to y$ as $n\to\infty$.
      \item if $y\in\BFM$, then 
      passing to a subsequence of $x_n$,
      there exists  a sequence $u_n\in U$ such that $x_nu_n\in\RFM$ 
      and $x_nu_n$ converges to an element of $\BFM$ as $n \to \infty$, which is potentially different from $y$. 
  \end{enumerate}
\end{proposition}
\begin{proof}
    This was proved in \cite{McMullen2016horocycles} under an extra assumption that $\M$ is convex cocompact. On the other hand, the proof only uses the fact that $\La$ is a round Sierpi\'nski gasket.  
    Therefore, the same proof works verbatim in our setting.
\end{proof}

\subsection*{Volume of horospheres}

Since every parabolic limit point is of rank two (Lemma \ref{lem.onlyrank2}), every compact $U$-orbit in $\Ga \ba G$ is contained in a compact $N$-orbit (Lemma \ref{lem.ferte}). This observation is useful in our classification of closures of  horocycles.

Let $x \in \Ga \ba G$ be such that $xN$ is compact. We denote by $\V(x)$ the volume of $xN$ with respect to the Haar measure of $N$.
We then have
\begin{equation}\label{eq.scale}
    \V(xa_tn)=e^{-2t}\V(x)\text{ for all } t\in\bb R,\  n \in N.
\end{equation}

For each $\xi > 0$, we define the following closed subsets of $\Ga \ba G$: $$\begin{aligned}
    F_{\xi}(N) & := \{x \in \Ga \ba G : xN \mbox{ is compact and }\V(x) \le \xi\} \\
    F_{\xi} & := F_{\xi}(N) \cdot K.
\end{aligned}$$
Geometrically, the set $F_{\xi}(N)$ consists of frames tangent to horospheres in a horoball, while $F_{\xi}$ consists of all frames based in the horoball.
Note that 
$\pi^{-1}(\core(\M)) -\op{int}(F_\xi)$ is a compact subset of $\Ga \ba G$, 
since $\M$ is geometrically finite and every parabolic limit point is of rank 2 (Lemma \ref{lem.onlyrank2}).

\subsection*{Cusp neighborhoods}
For an (open) horoball $\horo$ in $\H^3$ and $\rho \ge 0$, let $\horo_\rho \subset \horo$ be the horoball in $\horo$ with distance $\rho$ from $\partial \horo$.
Since $\M$ is geometrically finite, there exists $\xi_{\M} >0$ and finitely many horoballs $\horo^1, \cdots, \horo^n \subset \H^3$ with disjoint closures so that
 \be \label{eqn.defxi0}
 \inte( \pi(F_{\xi_{\M}}) ) = \bigcup_{i = 1}^n \Ga \horo^i
 \ee
 where $\Ga \horo^i \subset \M$ is the image of $\horo^i$ under the quotient map $\H^3 \to \M$.
 Fixing such $\xi_{\M} > 0$ and $\horo^i$'s, we simply write
 $$\Horo := \inte( F_{\xi_{\M}} ) = \left\lbrace x \in \Ga \ba G: \pi(x) \in \bigcup_{i = 1}^n \Gamma \horo^i \right\rbrace$$
 For $\rho \ge 0$, we similarly define
 $$\Horo_{\rho} := \left\lbrace x \in \Ga \ba G: \pi(x) \in \bigcup_{i = 1}^n \Gamma \horo^i_{\rho} \right\rbrace.$$
Since every parabolic limit point is of rank two (Lemma \ref{lem.onlyrank2}), it follows from \eqref{eq.scale} that 
\begin{equation}\label{eq.horo}
\Horo_\rho= \inte ( F_{e^{-2\rho}\xi_{\M}} ).   
\end{equation}
We also define $$W_{\rho} := \RF \M - \Horo_\rho$$ 
which is a compact subset of $\Ga \ba G$.
Note that $W_0 = \RF \M - \Horo$.

%%%%%%%%%%%%%%%%%%%%%%%%%%%%%%%%%%%%%%%%%%%%
%%%%%%%%%%%%%%%%%%%%%%%%%%%%%%%%%%%%%%%%%%%%
%%%%%%%%%%%%%%%%%%%%%%%%%%%%%%%%%%%%%%%%%%%%
%%%%%%%%%%
%%%%%%%%%% Recurrence of horocycle
%%%%%%%%%%
%%%%%%%%%%%%%%%%%%%%%%%%%%%%%%%%%%%%%%%%%%%%
%%%%%%%%%%%%%%%%%%%%%%%%%%%%%%%%%%%%%%%%%%%%
%%%%%%%%%%%%%%%%%%%%%%%%%%%%%%%%%%%%%%%%%%%%

\section{Recurrence of horocycle flows and unipotent blowup} \label{sec.fn}

Let $\Ga < G$ be a geometrically finite  Kleinian group with a round Sierpi\'nski limit set and $\M = \Ga \ba \H^3$.
In this section, we discuss recurrence of horocycle flows on $\FM$, i.e., $U$-action on $\Ga \ba G$. We also collect lemmas concerning the sequence of non-trivial elements in the double coset space $S \ba G /U$ that converges to the identity coset where $S =U$, $N$, or $H$.
In \cite{McMullen2017geodesic}, they were referred to as ``unipotent blowup".

\subsection*{Recurrence of horocycle flows}
To study the recurrence, we use  the notion of thickness.

\begin{definition}[Thickness]
Let $k > 1$ and $T \subset \R$. We say that
\begin{itemize}
    
        \item $T$ is \emph{$k$-thick} if 
        $$
        T \cap ([-ks, -s] \cup [s, ks]) \neq \emptyset \quad \text{for all } s > 0;
        $$
        \item $T$ is \emph{$k$-thick at $\infty$} if there exists $s_T \ge 0$ such that
        $$
        T \cap ([-ks, -s] \cup [s, ks]) \neq \emptyset \quad \text{for all } s > s_T. 
        $$
    \end{itemize}
    \end{definition}
    Recall from \eqref{eqn.notation} that 
$$U = \left\{ u_t = \begin{pmatrix} 1 & t \\ 0 & 1 \end{pmatrix} : t \in \R \right\} \quad \text{and} \quad V = \left\{ v_t = \begin{pmatrix} 1 & i t \\ 0 & 1 \end{pmatrix} : t \in \R \right\}.
$$
Via the maps $u_t \leftrightarrow t$ and $v_t \leftrightarrow t$, we identify $U$ and $V$ with $\R$ and define the thickness of subsets of $U$ and $V$ as well: $T \subset U$ is called $k$-thick (resp. $k$-thick at $\infty$) if $\{t \in \R : u_t \in T \}$ is $k$-thick (resp. $k$-thick at $\infty$). Similarly, $T \subset V$ is called $k$-thick (resp. $k$-thick at $\infty$) if $\{t \in \R : v_t \in T \}$ is $k$-thick (resp. $k$-thick at $\infty$). 

We use the term ``thickness" for subsets of $\R$, $U$, and $V$, with the specific choice made for convenience in each context.
In many cases, we measure the thickness of the recurrence time for horocycle flows.

\begin{definition}[Recurrence time] \label{def.tx}
    For $x\in\Ga\ba G$ and $W\subset \Ga\ba G$, set
    $$
    \mathsf{T}_W(x):=\{u \in U : xu \in W\}
    $$
\end{definition}
The following recurrence of horocycle flows was established by McMullen-Mohammadi-Oh:
\begin{prop} \cite[Lemma 9.2]{McMullen2017geodesic}\label{prop.tt}
    There exists $k_{\M} > 1$ depending only on $\M$ such that for all $x \in \RFM$, the set $\mathsf{T}_{\RFM}(x)$ is $k_{\M}$-thick.
\end{prop}
\begin{proof}
    The proposition is stated in \cite{McMullen2017geodesic} for the case when $\Ga$ is further assumed to be convex cocompact.
    On the other hand, their proof only relies on the fact that $\La$ is a round Sierpi\'nski gasket. Therefore, the same proof works in our setting.
\end{proof}
For practical applications, it is desirable for the set $W$ in Definition \ref{def.tx} to be compact.
Recall that $W_\rho=\RFM-\Horo_\rho$ is compact for all $\rho\geq 0$.
When $\M$ admits a cusp and hence $\RFM$ is not compact, $W_\rho$ will replace the role of $\RFM$ in Proposition \ref{prop.tt}.
We deduce the following from the work of Benoist-Oh, where general geometrically finite, acylindrical Kleinian groups were considered.

\begin{proposition} \cite[Corollary 5.5, Proposition 5.4]{Benoist2022geodesic}
 \label{prop.thickness}
  Let $\xi_{\M} > 0$ and $k_{\M}>1$ be as in \eqref{eqn.defxi0} and Proposition \ref{prop.tt} respectively.
  Then there exists $R>0$ such that the following holds:
  \begin{enumerate}
      \item for any $\rho\geq 0$ and $x \in  W_{\rho}$, the set $\mathsf{T}_{ W_{\rho+R}}(x)$ is $4k_{\M}$-thick.
      \item for any $\rho\geq 0$ and $x \in  W_{\rho+R}$, the set $\mathsf{T}_{W_{\rho+R}}(x)$ is $4k_{\M}$-thick at $\infty$.
      \item for any $x \in  \RFM-F_{e^{-2R}\xi_{\M}}(N)$, the set $\mathsf{T}_{W_R}(x)$ is $4k_{\M}$-thick at $\infty$.
 \end{enumerate}
\end{proposition}
\begin{proof}
    Let us explain how the statement can be deduced from \cite{Benoist2022geodesic}, accounting for the differences in formulation.
    In \cite[Corollary 5.5, Proposition 5.4]{Benoist2022geodesic}, they considered the set
    $$\RFKM := \left\{ x \in \RFM : \begin{matrix}
        \exists \ \T' \subset \T_{\RFM}(x) \text{ s.t.} \\
        e \in \T' \text{ and } \T' \cdot u^{-1} \text{ is } k\text{-thick} \ \forall u \in \T'\end{matrix}\right\}$$
    for each $k > 1$ and showed that for some $R > 0$,
    \be \begin{aligned} \label{eqn.applyBO}
        & \T_{\RFKM - \Horo_R}(x) \text{ is } 4k\text{-thick for all } x \in \RFKM - \Horo; \\
        & \T_{\RFKM - \Horo_R}(x) \text{ is } 4k\text{-thick at } \infty \text{ for all } x \in \RFKM - \Horo_R.
    \end{aligned}\ee

    Now we take $k=k_{\M}$ given by Proposition \ref{prop.tt}. Then for any $x \in \RFM$, $e \in \T_{\RFM}(x)$. Moreover, for any $u \in \T_{\RFM}(x)$, $x u \in \RFM$ and $\T_{\RFM}(xu) = \T_{\RFM}(x)\cdot u^{-1}$. Hence, it follows from Proposition \ref{prop.tt} that $\T_{\RFM}(x) \cdot u^{-1}$ is $k_{\M}$-thick for all $u \in \T_{\RFM}(x)$. This verifies that $\RFKM = \RFM$ with $k=k_{\M}$ in our setting, and therefore (1) and (2) for $\rho = 0$ follow from \eqref{eqn.applyBO}, noting that $W_0 = \RFM - \Horo$ and $W_R = \RFM - \Horo_R$. In fact, only the fact that $\Horo$ consists of disjoint horoballs was used in the proof of \cite[Corollary 5.5, Proposition 5.4]{Benoist2022geodesic}. Hence, we can replace $\Horo$ and $\Horo_R$ with deeper horoballs $\Horo_\rho \subset \Horo$ and $\Horo_{\rho + R} \subset \Horo_R$ for arbitrary $\rho \ge 0$ in \eqref{eqn.applyBO}. Therefore, (1) and (2) hold for general $\rho \ge 0$.

    As for item (3), recall that $\Horo= \inte( F_{\xi_{\M}} )$ and $\Horo_R= \inte(F_{e^{-2R}\xi_{\M}})$ from the identification we made in \eqref{eq.horo}.
    Hence in terms of $F_{\xi_{\M}}$, item (2) with the choice of $\rho=0$ translates into the statement that for $R>0$ given in the previous statement, $\mathsf{T}_{W_R}(x)$ is $4k_{\M}$-thick at $\infty$ for all $x\in \RFM- \inte(F_{e^{-2R}\xi_{\M}})$.
    In the proof of  \cite[Proposition 5.4]{Benoist2022geodesic}, the condition that $x \notin \inte(F_{e^{-2R}\xi_{\M}})$ was used to have that $x U$ is not contained in $\inte(F_{e^{-2R}\xi_{\M}})$. However, it is enough to have $x \notin F_{e^{-2R}\xi_{\M}}(N)$ to guarantee that $xU$ is not contained in $F_{e^{-2R}\xi_{\M}}$. Therefore, the argument therein works and (3) follows.
\end{proof}

\subsection*{Unipotent blowup: thick sets}

The notion of topological limsup will be repeatedly used throughout the paper. We mainly consider a sequence of sets parametrized by a subset of $\R$. 

\begin{definition}\label{def.limsup1}
    Let $\cal X$ be a metric space.
    For $T\subset\bb R$ and family of subsets $\{ Y_t\subset \cal X : t \in T\}$, we define
    $$
    \limsup_{t\in T,   t \to \infty} Y_t:=\left\{x\in \cal X: 
        \begin{array}{c}
        \text{$\exists$ sequences $t_n \in T$ and $y_{t_n} \in Y_{t_n}$}\\
        \text{s.t. $t_n \to \infty$ and $y_{t_n}\to x$ as $n\to\infty$.}
        \end{array}
    \right\}    
    $$
    In other words,
    \be \label{eqn.deflimsup}
    \limsup_{t \in T, t \to \infty} Y_t = \bigcap_{t_0 > 0} \overline{ \bigcup_{t \in T, t > t_0} Y_t }.
    \ee
    When $T = \N$ or  $(t_0, \infty) \subset T$ for some $t_0 \in \R$, we simply write $\limsup_{n \to \infty} Y_n$ or $\limsup_{t \to \infty} Y_t$ respectively.
\end{definition}

We use the above definition for $\cal X = G$ or $\cal X = \Ga \ba G$. Note that by \eqref{eqn.deflimsup}, the topological limsup is closed. We first record unipotent blowup lemmas that take thick sets into account.

\begin{lemma} {\cite[Lemma 6.1]{Benoist2022geodesic}} \label{lem.recurrence}
  Let $T \subset U$ be $k$-thick at $\infty$ for some $k > 1$. If $g_n \in G - AN$ is a sequence such that $g_n \to e$, then $\limsup_{n \to \infty} Tg_nU$ contains a sequence $\ell_n \to e$ in $AV - \{e \}$.
\end{lemma}

\begin{lemma} {\cite[Theorem 3.1]{McMullen2016horocycles}} \label{lem.seqrecurrence}
    Let $T_n \subset U$ be a sequence of $k$-thick sets for some $k > 1$. If $g_n \in G - HV$ is a sequence such that $g_n \to e$, then there exists $k'>1$ depending only on $k$, and a $k'$-thick set $V_0 \subset V$ such that $$V_0 \subset \limsup_{n\to \infty} Hg_n T_n .$$
\end{lemma}
 
\subsection*{Unipotent blowup: polynomials}
We now discuss the unipotent blowup lemma involving some polynomials, which was studied by Dani-Margulis \cite{Dani1989} and Shah \cite{Shah1992master}. For simplicity, we use the following notations: for $z \in \C - \{0\}$ and $s \in \R$,
$$d(z) = \begin{pmatrix}
z^{-1} & 0 \\
0 & z
\end{pmatrix} \in AM \quad \mbox{and} \quad
v(s) = 
\begin{pmatrix}
  1 & is \\
  0 & 1
\end{pmatrix} \in V.$$
Using the notations in \eqref{eqn.notation}, $d(z) = a_{- 2 \log |z|} a_{- 2 i \arg(z)}$ and $v(s) = v_s$.
We will consider real polynomials $\sigma, \nu \in \R[t]$ and $d(\sigma(t))v(\nu(t)) \in AV$, or a complex polynomial $\sigma \in \C[t]$ and $d(\sigma(t)) \in AM$, for $t \in \R$ such that $\sigma(t) \neq 0$.
The following is the unipotent blowup with polynomials:

\begin{lemma} \cite[Proposition 4.3.2]{Shah1992master} \label{lem.qr2}
    Let $S < G$ be either $U$ or $N$. Let $g_n \in G - \op{N}_G(S)$ be a sequence converging to the identity $e \in G$.
    \begin{enumerate}
        \item If $S = U$, then there exist polynomials $\sigma, \nu \in \R[t]$ such that at least one of them is non-constant, $\sigma(0) = 1$, $\nu(0) = 0$, and
        $$
        d(\sigma(t)) v(\nu(t)) \in \limsup_{n \to \infty} U g_n U \quad \text{for all } t \in \R \text{ with } \sigma(t)  \neq 0.
        $$

        \item If $S = N$, then there exists a non-constant polynomial $\sigma \in \C[t]$ such that $\sigma(0) = 1$ and 
        $$
        d(\sigma(t)) \in \limsup_{n \to \infty} N g_n U \quad \text{for all } t \in \R \text{ with } \sigma(t) \neq 0.
        $$
    \end{enumerate}
\end{lemma}

\begin{proof}
    As our formulation is slightly different from \cite[Proposition 4.3.2]{Shah1992master}, let us explain how we deduce the desired statement. Given a subgroup $S < G$ which is either $U$ or $N$ and a sequence $g_n \to e$ in $G - \op{N}_G(S)$, what directly follows from \cite[Proposition 4.3.2]{Shah1992master} is that
    \begin{enumerate}
        \item if $S = U$, then there exist polynomials $\sigma, \nu \in \R[t]$ such that at least one of them is non-constant, $\sigma(0) = 1$, $\nu(0) = 0$, and
        $$
        d(\sigma(t)) v(\nu(t)) \in \ov{\bigcup_{n \in \N} U g_n U} \quad \text{for all } t \in \R \text{ with } \sigma(t)  \neq 0.
        $$

        \item if $S = N$, then there exists a non-constant polynomial $\sigma \in \C[t]$ such that $\sigma(0) = 1$ and 
        $$
        d(\sigma(t)) \in  \ov{\bigcup_{n \in \N} N g_n U} \quad \text{for all } t \in \R \text{ with } \sigma(t) \neq 0.
        $$
    \end{enumerate}
    We handle both cases simultaneously by setting
    $$\Phi(t) := \begin{cases}
    d(\sigma(t)) v(\nu(t))  & \text{if }S = U \\
    d(\sigma(t)) & \text{if } S = N
    \end{cases}$$
    for $t \in \R$ with $\sigma(t) \neq 0$, where $\sigma$ and $\nu$ are polynomials given above.

    Then it suffices to deduce that for each $t \in \R$ with $\sigma(t) \neq 0$,
    \be \label{eqn.deduceshahconc}
    \Phi(t) \in \limsup_{n \to \infty} S g_n U.
    \ee
    Suppose not. Then  $\Phi(t) \notin \limsup_{n \to \infty} S g_n U$ for some $t \in \R$ with $\sigma(t) \neq 0$. Since $\Phi(t) \in \ov{\bigcup_{n \in \N} S g_n U}$, there exists $n \in \N$ and sequences $s_k \in S$ and $u_k \in U$ such that 
    $$
    \Phi(t) = \lim_{k \to \infty} s_k g_n u_k.
    $$
    Since $g_n \notin \op{N}_G(S)$, there exists $ s \in S$ such that $g_n  s g_n^{-1} \notin S$. Noting that $S$ and $U$ commute, we have
    \be \label{eqn.deduceshah}
    \lim_{k \to \infty} s_k (g_n s g_n^{-1}) s_k^{-1} = \lim_{k \to \infty} (s_k g_n u_k) s (s_k g_n u_k)^{-1} = \Phi(t) s \Phi(t)^{-1} \in S
    \ee
    since $\Phi(t) \in \op{N}_G(S)$. 
    
    For some $a, b, c, d \in \C$ and a sequence $c_k \in \C$, we write
    $$
    g_n s g_n^{-1} = \begin{pmatrix}
a & b \\
c & d
\end{pmatrix} \quad \text{and} \quad s_k = \begin{pmatrix}
1 & c_k \\
0 & 1
\end{pmatrix} \quad \text{for all } k \in \N.
    $$
    We then have
    $$
     s_k (g_n s g_n^{-1}) s_k^{-1} = \begin{pmatrix}
    a + c c_k & b - ac_k + dc_k - c c_k^2 \\
    c & d - c c_k
\end{pmatrix}
    $$
    which converges to an element in $S$ by \eqref{eqn.deduceshah}. This implies $c = 0$, and hence
     $$
     s_k (g_n s g_n^{-1}) s_k^{-1} = \begin{pmatrix}
    a  & b + c_k(d-a)  \\
    0 & d
\end{pmatrix}.
$$
Again, since this sequence converges to an element in $S$, we must have $a = d$, and hence
$$
     s_k (g_n s g_n^{-1}) s_k^{-1} = \begin{pmatrix}
    a  & b  \\
    0 & a
\end{pmatrix} \quad \text{for all } k \in \N.
$$
Hence, we now have that the constant sequence $s_k (g_n s g_n^{-1}) s_k^{-1}$ converges to an element of $S$ as $k \to \infty$. In particular,
$$
s_k (g_n s g_n^{-1}) s_k^{-1} \in S \quad \text{for all } k \in \N.
$$
This implies $g_n s g_n^{-1} \in S$, which is a contradiction. Therefore, \eqref{eqn.deduceshahconc} follows.
\end{proof}

The following proposition was proved by Dani-Margulis when $\sigma$ is non-constant, and by Shah in general:
\begin{proposition}[{\cite[Proposition 2.4]{Dani1989},  \cite[Proposition 4.4.3]{Shah1992master}}]    \label{prop.blowup2}
  Let $\sigma, \nu \in \R[t]$ be polynomials such that one of them is non-constant. Let $t_0 \ge 0$ be such that $\sigma(t) \neq 0$ for all $t > t_0$. Defining a function $\Phi : (t_0, \infty) \to AV$ as 
  \be \label{eq.defphi}
  \Phi(t) := d(\sigma(t)) v(\nu(t)),
  \ee 
  there exists a non-trivial one-parameter subgroup $L < AV$, and for every $\ell \in L$, there exists a function $f_\ell : (0, \infty) \to \R$ such that as $t \to \infty$, $$t + f_\ell(t) \to \infty \quad \mbox{and} \quad \Phi(t)^{-1}\Phi(t + f_\ell(t)) \to \ell.$$
\end{proposition}

Recalling Definition \ref{def.limsup1}, we have:

\begin{corollary}\label{cor.blowup2}
Let $\Phi(t)=d(\sigma(t))v(\nu(t))$ be as in \eqref{eq.defphi}.
For any $Y\subset\Ga\ba G$, there exists a one-parameter subgroup $L<AV$ such that
$$
\limsup_{t\to \infty}Y\Phi(t) \text{ is invariant under } L.$$
\end{corollary}
\begin{proof}
    Let  $X_0:=\limsup_{t \to \infty}Y\Phi(t)$ and $x_0\in X_0$ be arbitrary.
    Then $x_0 = \lim_{n \to \infty} y_n\Phi(t_n)$ for some sequences $y_n\in Y$ and $t_n\to\infty$ as $n\to\infty$.
    We then apply Proposition \ref{prop.blowup2}: let $L < AV$ and $\{f_\ell : \ell \in L \}$ be as in the proposition associated to $\Phi(t)$.
    Then for every $\ell \in L$, $t_n + f_\ell(t_n) \to \infty$ and
    $$
    y_n\Phi(t_n+f_\ell(t_n)) = y_n\Phi(t_n)(\Phi(t_n)^{-1}\Phi(t_n+f_\ell(t_n))) \to x_0 \ell.
    $$ This implies that
    $$x_0 \ell = \lim_{n \to \infty} y_n\Phi(t_n+f_\ell(t_n)) \in X_0.$$
    Since $x_0\in X_0$ and $\ell \in L$ are arbitrary, we get $X_0L\subset X_0$, and hence $X_0L=X_0$.
\end{proof}

The following is standard (cf. \cite[Lemma 2.2.2]{Shah1992master}):

\begin{lemma}\label{lem.1psg}
    If $L$ is a one-parameter subgroup of $AV$, then either $L = V$ or $L=vAv^{-1}$ for some $v\in V$.
\end{lemma}

%%%%%%%%%%%%%%%%%%%%%%%%%%%%%%%%%%%%%%%%%%%%
%%%%%%%%%%%%%%%%%%%%%%%%%%%%%%%%%%%%%%%%%%%%
%%%%%%%%%%%%%%%%%%%%%%%%%%%%%%%%%%%%%%%%%%%%
%%%%%%%%%%
%%%%%%%%%% Horospheres
%%%%%%%%%%
%%%%%%%%%%%%%%%%%%%%%%%%%%%%%%%%%%%%%%%%%%%%
%%%%%%%%%%%%%%%%%%%%%%%%%%%%%%%%%%%%%%%%%%%%
%%%%%%%%%%%%%%%%%%%%%%%%%%%%%%%%%%%%%%%%%%%%

\section{Expansion of $N$-orbits} \label{sec.exphoro}

In this section, we discuss some properties of horospheres in a complete hyperbolic 3-manifold $\Ga \ba \H^3$, which correspond to $N$-orbits in $\FM = \Ga \ba G$. 
We begin with the classification of $N$-orbit closures. The following was proved by Ferte, and the last claim follows from Lemma \ref{lem.onlyrank2}.

\begin{lemma} \cite[Theorem A, Theorem B]{ferte_horosphere} \label{lem.ferte}
    Let $\Ga$ be a non-elementary Kleinian group and $\M = \Ga \ba \H^3$. Let $x = [g] \in \RFPM$ for $g \in G$. 
    \begin{enumerate}
        \item If $g^+$ is a conical limit point, then
        $$\ov{xN} = \RFPM.$$
        \item If $g^+$ is a parabolic limit point of rank 1, then
        $xN$ is closed and not compact. 
        \item If $g^+$ is a parabolic limit point of rank 2, then $xN$ is compact.
    \end{enumerate}
    In particular, if $\Ga$ is geometrically finite with a round Sierpi\'nski limit set, then for any $x \in \RFPM$, either (1) or (3) occurs.
\end{lemma}

One key feature of horospheres is the expansion along frame flows in negative time as in the following lemma, which is a consequence of the equidistribution result due to Winter \cite{Winter2015mixing}:

\begin{lemma} \cite[Theorem 6.1]{Winter2015mixing} \label{lem.eqd}
    Let $\Ga < G$ be a Zariski dense geometrically finite Kleinian group and $\M = \Ga \ba \H^3$. Let $x\in \RFPM$.
    For any sequence $t_n\to+\infty$ in $\bb R$ and $m_n\in M$, we have
    $$
    \limsup_{n\to\infty}xNm_na_{-t_n}=\RFPM.
    $$
\end{lemma}

In the rest of the section, let $\Ga < G$ be a geometrically finite Kleinian group with a round Sierpi\'nski limit set $\La$, and $\M = \Ga \ba \H^3$. Note that both Lemma \ref{lem.ferte} and Lemma \ref{lem.eqd} apply to $\M$.
From the expansion of horospheres above, the classification of $N$-invariant subsets follows:

\begin{proposition} \label{prop.classifyhoro}
    Let $X_0\subset\RFPM$ be a closed $N$-invariant set.
    Then either $X_0=\RFPM$, or there exists $\eta>0$ such that $X_0\subset F_\eta(N)$.
\end{proposition}
\begin{proof}
Suppose $X_0\neq\RFPM$.
By Lemma \ref{lem.ferte}, for $x = [g] \in \RFPM$ such that $g^+ \in \La$ is a conical limit point of $\Ga$, the orbit $xN \subset \RFPM$ is dense. Therefore, for any $x = [g] \in X_0$, we have that $g^+ \in \La$ is a parabolic limit point of $\Ga$. This implies that there are finitely many $z_1, \cdots, z_k \in \RFPM$ such that
$$X_0 \subset \bigcup_{i = 1}^k z_i NMA.$$

Suppose to the contrary that $X_0 \not\subset F_{n}(N)$ for all $n \ge 1$. Then there exists a sequence $x_n \in X_0 - F_n(N)$ for all $n \ge 1$. For each $n \ge 1$, write
$$
x_n = z_{i_n} p_n m_n a_{-t_n} \in z_{i_n} NMA.
$$
After passing to a subsequence, we may assume that $z_{i_n} = z$ is constant.
Since $\V(x_n) = e^{2 t_n} \V(z p_n m_n) = e^{2t_n} \V(z)$ by \eqref{eq.scale}, we must have $t_n \to + \infty$ as $n \to \infty$. We then have that $X_0$ contains $x_n N = z N m_n a_{-t_n}$ for all $n \ge 1$. Therefore, $\RFPM \subset X_0$ by Lemma \ref{lem.eqd}, which is a contradiction. This proves the lemma.
\end{proof}

\subsection*{$U$-orbits in expanding $N$-orbits}

We now consider a sequence of $U$-orbits contained in an expanding sequence of $N$-orbits. We will show that we can find a $v AU v^{-1}$-orbit for some $v \in V$ in the set of accummulation points of such $U$-orbits (Proposition \ref{prop.sigma2}). We first prove the following lemma:

\begin{lemma}\label{lem.xt}

    Let $x_n \in \RFM \cdot U$ be a sequence such that for any subsequence $\{ x_{n_j} : x_{n_j}N \text{ is compact} \}$, we have $\V(x_{n_j}) \to \infty$.
    Then for any $\eta>0$, there exists a neighborhood $O_{\eta}(N) \subset \Ga \ba G$ of $F_{\eta}(N)$ such that $\RFM - O_{\eta}(N)$ is compact and 
    $$
    \left( \limsup_{n \to \infty}  x_n U \right) \cap \RFM - O_{\eta}(N) \neq\emptyset.
    $$
    In particular, $\left( \limsup_{n \to \infty}  x_n U \right) \cap \RFM - F_{\eta}(N) \neq\emptyset$.
\end{lemma}

\begin{proof}
    Let $\xi_{\M}, R>0$ be as in \eqref{eqn.defxi0} and Proposition \ref{prop.thickness}, respectively, and $\xi_1:=e^{-2R}\xi_{\M}$.
    Given any $\eta>0$, choose $s>0$ satisfying $e^{-2s}\eta<0.5\xi_1$.

    \medskip
    
    We claim that $x_nUa_s \cap \RFM-F_{\xi_1}(N) \neq \emptyset$ for all sufficiently large $n \ge 1$.
    Note that by \eqref{eq.scale},
    $$
    x_n Ua_s\cap\RFM-F_{\xi_1}(N)=(x_n U\cap\RFM-F_{e^{2s}\xi_1}(N))a_s,
    $$
    and it suffices to check that $x_n U\cap\RFM-F_{e^{2s}\xi_1}(N)\neq\emptyset$.
    Since $x_n \in\RFM\cdot U$, there exists $\tilde x_n \in x_n U\cap\RFM$ for each $n \ge 1$.
    If $x_n N$ is not compact, then neither is $\tilde x_nN$, and hence $\tilde x_n \not\in F_{e^{2s}\xi_1}(N)$ trivially.
    Hence it suffices to consider the case that $x_n N$ is compact for infinitely many $n \ge 1$. If $x_n N$ is compact, then $\tilde x_n N$ is compact as well and $\V(\tilde x_n)=\V(x_n)$, which diverges as $n \to \infty$ by the hypothesis.
    Hence $\tilde x_n \not\in F_{e^{2s}\xi_1}(N)$ for all sufficiently large $n \ge 1$.
    This proves the claim.

    \medskip

    Note that $W_R = \RFM - \inte(F_{\xi_1})$ using notations in Proposition \ref{prop.thickness}.
    It then follows from  the above claim and Proposition \ref{prop.thickness}(3) that for $k_{M} > 1$ given in Proposition \ref{prop.tt} and all sufficiently large $n \ge 1$, there exists $y_n \in x_n U$ such that $\mathsf{T}_{W_R}(y_na_s)$ is $4k_{\M}$-thick at $\infty$. In particular,
    $$x_n Ua_s \cap W_R =y_nUa_s \cap W_R =y_na_sU \cap W_R \neq \emptyset.$$

    Denote by $O_\eta(N):=\op{int}(F_{\xi_1})a_{s}^{-1}$. 
    From the definition of $W_R$ and $\RFM=\RFM\cdot a_s$, it follows that
    \begin{equation}\label{eq.oeta}
    x_n U\cap\RFM -O_\eta(N)\neq\emptyset
    \end{equation}
    for all large enough $n \ge 1$.
    Note that $O_\eta(N)$ is an open neighborhood of $F_\eta(N)$ because
    $$
    \op{int}(F_{\xi_1})a_{s}^{-1}\supset F_{0.5\xi_1}a_{s}^{-1}\supset F_{0.5\xi_1}(N)a_{s}^{-1}=F_{0.5e^{2s}\xi_1}(N)\supset F_{\eta}(N)
    $$
    by \eqref{eq.scale}.
    Since $\RFM-O_\eta(N) = (\RFM - \inte (F_{\xi_1}))a_s^{-1}$ is compact, the lemma follows from \eqref{eq.oeta}.
\end{proof}

We now prove the existence of a $vAUv^{-1}$-orbit for some $v \in V$ mentioned above. Recall that for $z \in \C - \{0\}$ and $s \in \R$,
$$d(z) = \begin{pmatrix}
z^{-1} & 0 \\
0 & z
\end{pmatrix} \in AM \quad \mbox{and} \quad
v(s) = 
\begin{pmatrix}
  1 & is \\
  0 & 1
\end{pmatrix} \in V.$$

\begin{proposition}\label{prop.sigma2}

    Let $y \in \RFPM$.
    Let $\sigma, \nu \in \R[t]$ be polynomials and set $\Phi(t) = d(\sigma(t)) v(\nu(t))$ for $t \in \R$ with $\sigma(t) \neq 0$.
    Suppose that $\sigma$ is non-constant and $y\Phi(t)\in\RFM\cdot U$ for all sufficiently large $t > 0$. Then there exists $v \in V$ such that
    $$
    \limsup_{t \to \infty } yU\Phi(t) \text{ contains a } v AUv^{-1}\text{-orbit.}
    $$
\end{proposition}
\begin{proof}
    Let $X_0:=\limsup_{t \to \infty} yU\Phi(t)$. We first claim that for any $\eta > 0$,
    \be \label{eqn.applylem51}
    X_0 \cap \RFM - F_{\eta}(N) \neq \emptyset.
    \ee
    Fix a sequence $t_n \to \infty$ and for each $n \ge 1$, let $x_n := y \Phi(t_n) \in \RFM \cdot U$.
    If $x_{n_j}$ is a subsequence such that $x_{n_j}N$ is compact for all $n \ge 1$, then $yN$ is compact as well, and moreover it follows from \eqref{eq.scale} that 
    $$
    \V(x_{n_j}) = \V(y d(\sigma(t_{n_j}))) = |\sigma(t_{n_j})|^4 \cdot \V(y).
    $$
    Since $\sigma$ is non-constant and $\lim_{j \to \infty} t_{n_j} = \infty$, we have $|\sigma(t_{n_j})| \to \infty$, from which we deduce $\V(x_{n_j}) \to \infty$ as $j \to \infty$. Therefore, the sequence $x_n \in \RFM \cdot U$ satisfies the condition of  Lemma \ref{lem.xt}, and hence for any $\eta > 0$,
    $$\left( \limsup_{n \to \infty} x_n U \right) \cap \RFM - F_{\eta}(N) \neq \emptyset.$$
    Since $\Phi(t_n) \in AV < \op{N}_G(U)$,
    $$
    \limsup_{n \to \infty} y U \Phi(t_n) = \limsup_{n \to \infty} y \Phi(t_n) U =\limsup_{n \to \infty} x_n U,
    $$
    and therefore
    $$
    X_0 \cap \RFM - F_{\eta}(N) \supset \left( \limsup_{n \to \infty} y U \Phi(t_n) \right) \cap \RFM - F_{\eta}(N) \neq \emptyset.
    $$
    This shows the claim.

        \medskip

    Now by Corollary \ref{cor.blowup2}, $X_0$ is invariant under a one-parameter subgroup $L<AV$, and it follows from Lemma \ref{lem.1psg} that either $L = V$ or  $L = vAv^{-1}$ for some $v \in V$.
    Moreover, since $\Phi(t) \in AV < \op{N}_G(U)$, $X_0$ is $U$-invariant. Together with the commutativity of $U$ and $V$, we now have that
    $$
    X_0 \text{ is invariant under } N = UV \text{ or } vAUv^{-1}.
    $$

    Suppose first that $X_0$ is $N$-invariant. By  Proposition \ref{prop.classifyhoro}, we have $X_0 = \RFPM$ or $X_0 \subset F_{\eta}(N)$ for some $\eta  > 0$. 
    The latter is forbidden by the claim \eqref{eqn.applylem51}, and hence 
    $$X_0 = \RFPM$$
    in this case. Then $X_0$ is a non-empty $AU$-invariant set, and therefore $X_0$ contains an $AU$-orbit
    as desired.

    Now suppose that $X_0$ is invariant under $v AU v^{-1}$ for some $v \in V$. By \eqref{eqn.applylem51}, we in particular have that $X_0$ is non-empty. Therefore, there exists a $vAUv^{-1}$-orbit in $X_0$. This completes the proof.
\end{proof}

%%%%%%%%%%%%%%%%%%%%%%%%%%%%%%%%%%%%%%%%%%%%
%%%%%%%%%%%%%%%%%%%%%%%%%%%%%%%%%%%%%%%%%%%%
%%%%%%%%%%%%%%%%%%%%%%%%%%%%%%%%%%%%%%%%%%%%
%%%%%%%%%%
%%%%%%%%%% Geodesic planes
%%%%%%%%%%
%%%%%%%%%%%%%%%%%%%%%%%%%%%%%%%%%%%%%%%%%%%%
%%%%%%%%%%%%%%%%%%%%%%%%%%%%%%%%%%%%%%%%%%%%
%%%%%%%%%%%%%%%%%%%%%%%%%%%%%%%%%%%%%%%%%%%%

\section{A closed $H$-orbit and $U$-orbits therein} \label{sec.geodesicplane}

Let $\Ga < G$ be a geometrically finite Kleinian group with a round Sierpi\'nski limit set and $\M = \Ga \ba \H^3$.
We will also employ geometry and dynamics appearing in geodesic planes in $\M = \Ga \ba \H^3$, or $H$-orbits in $\FM = \Ga \ba G$. In this section, we discuss properties of a closed $H$-orbit and $U$-orbits therein.
Recall that $H < G$ is a copy of $\PSL_2(\R)$ which is an orientation-preserving stabilizer of $\widehat \R \subset \widehat \C$, the boundary of $\H^2$-copy in $\H^3$ invariant under $H$. The following is the classification of $H$-orbit closures by Benoist-Oh:

\begin{theorem} \cite[Theorem 11.10]{Benoist2022geodesic} \label{thm.plane}
    Let  $y \in\RFM$. Either 
    $$y H \text{ is closed} \quad \text{or} \quad  \ov{y H}=\RFPM\cdot H.$$
\end{theorem}
\begin{remark}
We note that Benoist-Oh showed Theorem \ref{thm.plane} in a more general setting that $\M$ is geometrically finite and acylindrical with $\partial \core(\M)$ totally geodesic.
\end{remark}

When $yH$ is closed and $y = [g]$ for $g \in G$, the conjugate $\Ga^g := g^{-1} \Ga g$ is the stabilizer of $y \in \Ga \ba G$ for the right-multiplication, and the orbit map
\be \label{eqn.properembedding}
    \begin{aligned}
        \phi: (H\cap \Ga^g)\ba H&\to \Ga\ba G\\
        (H\cap \Ga^g)h&\mapsto yh
    \end{aligned}
    \ee
    is a proper embedding \cite[Section 4.2]{Oh2013equidistribution}. Via $\phi$, the closed $H$-orbit $xH$ can be identified with the unit tangent bundle $(H \cap \Ga^g) \ba H$ of the hyperbolic surface $(H \cap \Ga^g) \ba \H^2$. In this regard, we recall the following, which is a special case of the work of Dal'bo:

    \begin{lemma} \cite[Proposition B]{Dalbo2000topologie} \label{lem.dalbo}
        Let $\Ga_H < H$ be a non-elementary discrete subgroup with the limit set $\La_{\Ga_H}$. Let $y = [h] \in \Ga_H \ba H$ be such that $h^+ \in \La_{\Ga_H}$.
        \begin{enumerate}
            \item If $h^+$ is a conical limit point of $\Ga_{H}$, then
            $$
            \ov{yU} = \{ z = [\ell] \in \Ga_H \ba H : \ell^+ \in \La_{\Ga_H} \}.
            $$
            \item If $h^+$ is a parabolic limit point of $\Ga_{H}$, then $yU$ is compact.
        \end{enumerate}
        In particular, if $\Ga_{H}$ is geometrically finite, then for any $y = [h] \in \Ga_H \ba H$ with $h^+ \in \La_{\Ga_H}$, either (1) or (2) occurs.
    \end{lemma}
    Note that this is a surface version of Lemma \ref{lem.ferte}. Although we can try to apply Lemma \ref{lem.dalbo} to $H \cap \Ga^g < H$, the limit set we are interested in is the limit set $\La$ of $\Ga$ or the limit set $g^{-1} \La$ of $\Ga^g$, not the limit set of $H \cap \Ga^g$. We first need to handle this subtlety.

\subsection*{$U$-orbits in a closed $H$-orbit}

We deduce from the work of Oh-Shah \cite{Oh2013equidistribution} that in our setting, the limit set of $H \cap \Ga^g$ is precisely equal to the intersection of the limit set of $\Ga^g$ with the circle $\widehat \R$ stabilized by $H$.

\begin{proposition} \label{prop.noneltgf}
    Let $y = [g] \in \RFM$ for some $g \in G$ be such that $yH$ is closed. Then $H \cap \Ga^g < H$ is a non-elementary geometrically finite subgroup and its limit set is equal to $g^{-1} \La \cap \widehat \R$.
\end{proposition}

\begin{proof}
     It was shown in \cite[Theorem 4.7]{Oh2013equidistribution} that $H \cap \Ga^g$ is a geometrically finite subgroup of $H$.  Moreover, it is clear that the limit set of $H \cap \Ga^g$ is contained in $g^{-1} \La \cap \widehat \R$. Since $y \in \RFM$, $g^{-1} \La \cap \widehat \R$ is perfect; otherwise, $ g^{-1} \La \cap \widehat \R$ has an isolated point $z$. Since $g^{-1}\La \cap \widehat \R \subset \widehat \C$ is contained in  a circle $\widehat{\R}$, this means that there exists an open segment $I \subset \widehat{\R}$ such that $I \cap g^{-1} \La = \{z\}$. Since components of $\widehat \C - g^{-1} \La$ are round open disks with mutually disjoint closures, there exists a component $B \subset \widehat \C - g^{-1} \La$ containing $I - \{z\}$. This implies $g^{-1} \La \cap \widehat \R - \{z\} \subset B$, contradicting to $y \in \RFM$.
     
     In particular, $g^{-1} \La \cap \widehat \R$ is uncountable. Since $g^{-1}\La$ is the union of conical limit points of $\Ga^g$ and countably many parabolic limit points of $\Ga^g$, $g^{-1} \La \cap \widehat \R$ contains infinitely many conical limit points of $\Ga^g$. 
    By \cite[Lemma 4.5]{Oh2013equidistribution}, all conical limit points of $\Ga^g$ in $ g^{-1} \La \cap \widehat \R$ are conical limit points of $H \cap \Ga^g$. Therefore, $H \cap \Ga^g$ is non-elementary.

    Finally, we show that the limit set of $H \cap \Ga^g$ is equal to $g^{-1} \La \cap \widehat \R$. Without loss of generality, we may assume that $0 \in g^{-1} \La \cap \widehat \R$ and it suffices to show that $0$ is in the limit set of $H \cap \Ga^g$. There are two cases:

    \begin{itemize}
        \item

    Suppose first that there are sequences $\varepsilon_n', \varepsilon_n > 0$ such that $\varepsilon_n, \varepsilon_n' \to 0$ as $n \to \infty$ and $\varepsilon_n, -\varepsilon_n' \notin g^{-1} \La$ for all $n \ge 1$. Then $g^{-1} \La \cap (-\varepsilon_n', \varepsilon_n)$ is compact. 

    \medskip
    
    We claim that $g^{-1} \La \cap (-\varepsilon_n', \varepsilon_n)$ has no isolated point for all $n \ge 1$. Suppose to the contrary that $z \in g^{-1} \La \cap (-\varepsilon_n', \varepsilon_n)$ is isolated. Then for some $\delta > 0$, open segments $(z-\delta, z)$ and $(z, z + \delta)$ in  $(-\varepsilon_n', \varepsilon_n)$ are disjoint from $g^{-1}\La$. Since $z \in g^{-1} \La$ and $g^{-1} \La \cap \widehat \R$ is uncountable, this implies that there are two distinct components $B_1, B_2$ of $\widehat \C - g^{-1}\La$ such that $(z -\delta, z) \subset B_1$ and $(z, z + \delta) \subset B_2$. On the other hand, $z \in \ov{B_1} \cap \ov{B_2}$, which is a contradiction to the hypothesis that $\La$, and hence $g^{-1} \La$, is a round Sierpi\'nski limit set.

\medskip

    By the above claim, $g^{-1} \La \cap (-\varepsilon_n', \varepsilon_n)$ is perfect for all $n \ge 1$. In particular, $g^{-1} \La \cap (-\varepsilon_n', \varepsilon_n)$ is uncountable, and hence contains a conical limit point of $\Ga^g$ for all $n \ge 1$. Since $\varepsilon_n, \varepsilon_n' \to 0$ as $n \to \infty$, this implies that there is a sequence of conical limit points of $\Ga^g$ in $g^{-1} \La \cap \widehat \R$ that converges to $0$. Since every conical limit point of $\Ga^g$ is a conical limit point of $H \cap \Ga^g$ \cite[Lemma 4.5]{Oh2013equidistribution},  it follows that $0$ is a limit point of $H \cap \Ga^g$, as desired.

\medskip

    \item Otherwise, there exists $\delta > 0$ such that at least one of the segments $[-\delta, 0]$ or $[0, \delta]$ is contained in $g^{-1} \La \cap \widehat \R$. This implies that there is a sequence of conical limit points of $\Ga^g$ in $g^{-1} \La \cap \widehat \R$ that converges to $0$. As in the previous case, it follows that $0$ is a limit point of $H \cap \Ga^g$.

\end{itemize}
    In any case, $0$ is a limit point of $H \cap \Ga^g$, finishing the proof.
\end{proof}

We are now able to apply Lemma \ref{lem.dalbo} to $H \cap \Ga^g$ and obtain the following:

\begin{corollary} \label{cor.uinsurface}
    Let $y \in \RFM$ be such that $yH$ is closed. Let $z = [g_z] \in yH \cap \RFPM$ for $g_z \in G$.
    \begin{itemize}
        \item If $g_z^+$ is a conical limit point of $\Ga$, then
        $$
        \ov{zU} = yH \cap \RFPM.
        $$
        \item Otherwise, $zU$ is compact.
    \end{itemize}
    In particular, for any $y \in \BFM$, we have $\ov{yU} = yH$.
\end{corollary}

\begin{proof}
    We first prove the claim for $z = y$. Let $g \in G$ be such that $y = [g]$.
By Proposition \ref{prop.noneltgf}, $H \cap \Ga^g < H$ is a non-elementary geometrically finite subgroup and its limit set is equal to $g^{-1} \La \cap \widehat \R$. Since $e^+ = \infty \in g^{-1} \La \cap \widehat \R$, we can apply Lemma \ref{lem.dalbo} to the identity coset $[e] \in (H \cap \Ga^g) \ba H$. As mentioned in the proof of Proposition \ref{prop.noneltgf}, $e^+$ is a conical limit point of $H \cap \Ga^g$ if and only if $g^+$ is a conical limit point of $\Ga$ by \cite[Lemma 4.5]{Oh2013equidistribution}. Hence, applying Lemma \ref{lem.dalbo}, we obtain the following dichotomy:
\begin{itemize}
    \item if $g^+$ is a conical limit point of $\Ga$, then
    \be \label{eqn.beforephi}
    \ov{[e]U} = \{ z = [\ell] \in (H \cap \Ga^g) \ba H : \ell^+ \in g^{-1} \La \cap \widehat \R \}.
    \ee
    \item otherwise, $[e]U$ is compact.
\end{itemize}
Recall the proper embedding $\phi : (H \cap \Ga^g) \ba H \to \Ga \ba G$ from \eqref{eqn.properembedding}. Since $\phi([e]U) = yU$ and the right hand side of \eqref{eqn.beforephi} has the image $yH \cap \RFPM$ under $\phi$, the claim follows for $y$.

Now let $z \in yH \cap \RFPM$ be arbitrary. Since $yH$ meets $\RFM$, there exists $u \in U$ such that $zu \in \RFM$. We then have that $zuH = yH$ is closed, and hence the above claim applies to $zu$. Since $zuU = zU$, this finishes the proof.
\end{proof}

The following $AU$-minimality is a direct consequence of Proposition \ref{prop.noneltgf}:

\begin{corollary} \label{cor.auminimal}
    Let $y \in \RFM$ be such that $yH$ is closed. For any $z \in yH \cap \RFPM$,
    $$\ov{zAU} = yH \cap \RFPM.$$
\end{corollary}

\begin{proof}
    We first prove the claim for $z = y$. 
Let $g \in G$ be such that $y = [g]$. By Proposition \ref{prop.noneltgf}, $H \cap \Ga^g < H$ is non-elementary and geometrically finite, and its limit set is $g^{-1}\La \cap \widehat \R$. Since $g^+ \in \La$, $e^+ = \infty$ is contained in the limit set of $H \cap \Ga^g$. As $H \cap \Ga^g$ is non-elementary, it acts minimally on its limit set, and hence
$$\ov{(H \cap \Ga^g) e^+} = g^{-1}\La \cap \widehat \R.$$
 Since $\widehat \R = H/AU$, the above identity is equivalent to the following identity in $(H \cap \Ga^g) \ba H$:
 $$
    \ov{[e]AU} = \{ [h] \in (H \cap \Ga^g) \ba H : h^+ \in g^{-1} \La \cap \widehat \R \}.
 $$
 As in the proof of Corollary \ref{cor.uinsurface}, taking $\phi$ implies the claim for $y$.
The claim for general $z \in yH \cap \RFPM$ can be deduced by the same argument as in the proof of Corollary \ref{cor.uinsurface}.
\end{proof}

\subsection*{Expansion of $U$-orbits within a closed $H$-orbit}
In the rest of this section, we discuss expanding behaviors of compact $U$-orbits in a closed $H$-orbit. The following may be standard, and can be shown by arguments in the proof Proposition \ref{prop.classifyhoro}:

\begin{lemma} \label{lem.expansioninplane}
    Let $\Ga_H < H$ be a non-elementary geometrically finite subgroup with the limit set $\La_{\Ga_H}$. Let $y_n \in \Ga_H \ba H$ be a sequence such that $y_nU$ is compact for all $n \ge 1$ and the length of $y_n U$ with respect to the Haar measure of $U$ diverges as $n \to \infty$. Then
    $$
    \limsup_{n \to \infty} y_n U = \{ [h] \in \Ga_H \ba H : h^+ \in \La_{\Ga_H} \}.
    $$
\end{lemma}

\begin{proof}
    We sketch the argument. Since $\Ga_H$ is geometrically finite, there are finitely many elements $z_1, \cdots, z_k \in \Ga_H \ba H$ such that all compact $U$-orbits are contained in the union $\bigcup_{i = 1}^k z_i UA$. After passing to a subsequence, we may assume that $y_n \in z_1 UA$ for all $n \ge 1$, and hence there exists a sequence $t_n \in \R$ such that $y_n U = z_1 U a_{t_n}$ for all $n \ge 1$. Since the length of $y_n U$ diverges, we must have $t_n \to - \infty$ as $n \to \infty$, by the scaling property of the $U$-Haar measure along the geodesic flow, which is similar to \eqref{eq.scale}. Then the equidistribution result of Winter \cite[Theorem 6.1]{Winter2015mixing} 
    applies and finishes the proof as in Proposition \ref{prop.classifyhoro}.
\end{proof}

We now obtain the following expansion of compact $U$-orbits within a closed $H$-orbit:

\begin{lem}\label{lem.OS}
    Let $y \in \RFM$ be such that $yH$ is closed. Let $\eta_n \in \R$ be a sequence such that $\eta_n \to \infty$ as $n \to \infty$.
    Let $y_n \in yH$ be a sequence such that $y_n U$ is compact and $y_n \notin F_{\eta_n}(N)$ for all $n \ge 1$.
    Then 
    $$
    \limsup_{n \to\infty}y_nU=yH\cap\RFPM.
    $$
\end{lem}
\begin{proof} 
    Let $g \in G$ be such that $y=[g]\in\Ga\ba G$  and let $\Ga^g :=g^{-1}\Ga g$ be its stabilizer.
    By Proposition \ref{prop.noneltgf}, $H\cap\Ga^g$ is a non-elementary geometrically finite subgroup of $H$ and its limit set is $g^{-1} \La \cap \widehat \R$. Recall that the orbit map
    \begin{align*}
        \phi: (H\cap \Ga^g)\ba H&\to \Ga\ba G\\
        (H\cap \Ga^g)h&\mapsto yh
    \end{align*}
    is a proper embedding \cite[Section 4.2]{Oh2013equidistribution}.

    For each $n \ge 1$, let $z_n \in ( H \cap \Ga^g ) \ba H$ be such that $\phi(z_n) = y_n$. Since $y_n U$ is compact and $\phi$ is proper, $z_n U \subset (H \cap \Ga^g) \ba H$ is compact as well. Since $y_n \not\in F_{\eta_n}(N)$ and $\eta_n \to \infty$, the length of $y_n U$ diverges as $n \to \infty$, and hence the length of $z_nU$ does so. We now apply Lemma \ref{lem.expansioninplane} to the sequence $z_n \in (H \cap \Ga^g) \ba H$. Since the limit set of $H \cap \Ga^g < H$ is $g^{-1} \La \cap \widehat \R$, it follows from Lemma \ref{lem.expansioninplane} that
    $$
    \limsup_{n \to \infty} z_n U = \{ [h] \in (H \cap \Ga^g) \ba H : h^+ \in g^{-1} \La \cap \widehat \R \}.
    $$
    Therefore, applying $\phi$ finishes the proof.
\end{proof}

%%%%%%%%%%%%%%%%%%%%%%%%%%%%%%%%%%%%%%%%%%%%
%%%%%%%%%%%%%%%%%%%%%%%%%%%%%%%%%%%%%%%%%%%%
%%%%%%%%%%%%%%%%%%%%%%%%%%%%%%%%%%%%%%%%%%%%
%%%%%%%%%%
%%%%%%%%%% X with H
%%%%%%%%%%
%%%%%%%%%%%%%%%%%%%%%%%%%%%%%%%%%%%%%%%%%%%%
%%%%%%%%%%%%%%%%%%%%%%%%%%%%%%%%%%%%%%%%%%%%
%%%%%%%%%%%%%%%%%%%%%%%%%%%%%%%%%%%%%%%%%%%%

\section{A $U$-orbit closure containing a closed $H$-orbit}\label{sec.xwh}
Let $\Ga < G$ be a geometrically finite Kleinian group with a round Sierpi\'nski limit set and $\M = \Ga \ba \H^3$.
    The goal of this section is to classify a $U$-orbit closure $X=\ov{xU}$ in the case when $X$ contains a closed $H$-orbit meeting $\RFM$.
    
\begin{theorem} \label{thm.xwh}

Let $x \in \RFM$ and $X = \ov{xU}$. Suppose that there exists $y \in \RFM$ such that $yH$ is closed and $yH \cap \RFPM \subset X$. Then either
$$X = yH \cap \RFPM \quad \text{or} \quad X = \RFPM.$$

\end{theorem}

Recall the notion of boundary frames from \eqref{eqn.bfmdef}. We first make the following observation on the dichotomy of closed $H$-orbits:

\begin{lemma} \label{lem.closedH}
Let $y \in \Ga \ba G$ be such that $yH$ is closed. Then either 
$$yH \subset \BFM \quad \text{or} \quad yH \cap \BFM \cdot V = \emptyset.$$
\end{lemma}

\begin{proof}
Suppose that the closed $H$-orbit $yH$ intersects $\BFM \cdot V$. We then have $y_0 \in yH$, $z \in \BFM$, and $v \in V$ such that $$y_0 = zv.$$
For $t > 0$, we have
$
y_0 a_t = z a_t (a_t^{-1} v a_t)
$. Since $z a_t$ belongs to a compact set $\BFM$, there exists a sequence $t_n \to \infty$ as $n \to \infty$ so that $z a_{t_n}$ converges to a point in $\BFM$. We denote by $z_0 \in \BFM$ its limit. Since $a_t^{-1} v a_t \to e$ as $t \to \infty$,
$$
y_0 a_{t_n} = z a_{t_n} (a_{t_n}^{-1} v a_{t_n}) \to z_0 \in \BFM.
$$
On the other hand, $y_0 a_{t_n} \in yH$ and $yH$ is closed, and hence $z_0 \in yH$ as well. Since $\BFM$ is $H$-invariant, we  have
$$
yH = z_0 H \subset \BFM.
$$
This finishes the proof.
\end{proof}

The following is the key lemma of this section:

\begin{lemma}\label{lem.yvn}
    Let $y \in \RFM$ be such that $yH$ is closed. Let $v_n \in V$ be a sequence such that $v_n \to \infty$ as $n \to \infty$. Suppose that  $yHv_n\cap \RFM\cdot U\neq\emptyset$ for all $n \ge 1$. Then
    $$
    \limsup_{n\to\infty} (yHv_n \cap \RFPM) = \RFPM.
    $$
\end{lemma}
\begin{proof}
    For simplicity, we set $Y := yH \cap \RFPM$ which is $AU$-invariant and $X_0 := \limsup_{n \to \infty} Y v_n$. Then the lemma is equivalent to $X_0 = \RFPM$.

Since $Y$ is $AU$-invariant and $U$ and $V$ commute,
$$
X_0 = \limsup_{n \to \infty} Y v_n (v_n^{-1} A v_n) U.
$$
Since $V = \limsup_{n \to \infty} v_n^{-1} A v_n$ and $N = VU$, we have that
$$
X_0  = X_0 \cdot N.
$$

    By \eqref{eqn.deflimsup}, $X_0$ is closed as well. Now it follows from Proposition \ref{prop.classifyhoro} that either
    \be \label{eqn.plane1}
    X_0 = \RFPM \quad \text{or} \quad X_0 \subset F_{\eta}(N) \quad \text{for some } \eta > 0.
    \ee
    We apply Lemma \ref{lem.xt} to finish the proof, by finding a sequence $y_n \in Y$ such that $y_n v_n \in Y v_n$ satisfies the condition therein. By Lemma \ref{lem.closedH}, there are two cases:
    $$yH \subset \BFM \quad \text{or} \quad yH \cap \BFM \cdot V = \emptyset.$$

    Suppose first that $yH \subset \BFM$. In this case, for each $n \ge 1$, we choose any element $y_n \in yH$ such that $y_n v_n \in \RFM \cdot U$, which exists by the hypothesis. This in particular implies $y_n \in \RFPM$ as well, and hence $y_n \in Y$.
    Since $y_n \in \BFM$, $y_n H$ is compact, and hence $y_n v_n N = y_n N$ is not compact for all $n \ge 1$ by Lemma \ref{lem.ferte}. Therefore, the sequence $y_n v_n \in \RFM \cdot U$ satisfies the condition in Lemma \ref{lem.xt}.

    We now consider the case that $yH \cap \BFM \cdot V = \emptyset$. In this case, we fix any sequence $t_n  > 0$ such that $t_n  \to \infty$ as $n \to \infty$, and set $y_n := y a_{t_n}^{-1} \in Y$ for each $n \ge 1$, where $a_{t_n} = \begin{pmatrix} e^{t_n/2} & 0 \\ 0 & e^{-t_n/2} \end{pmatrix}$. Since $Y \cap \BFM \cdot V = \emptyset$, we also have that $y_n v_n \notin \BFM \cdot V$, and hence $y_n v_n \in \RFM \cdot U$ for all $n \ge 1$ by Lemma \ref{lem.zv}. If $y_n v_n N = y_n N$ is compact, then
    $$
    \V(y_n v_n) = \V(y a_{t_n}^{-1} v_n) = e^{2t_n} \V(y).
    $$
    Since $t_n \to \infty$ as $n \to \infty$, the sequence $y_n v_n \in \RFM \cdot U$ also satisfies the condition in Lemma \ref{lem.xt}. 

    In any case, we obtain a sequence $y_n v_n \in Y v_n \cap \RFM \cdot U$ to which Lemma \ref{lem.xt} applies. Therefore, we have for any $\eta > 0$ that
    $$
    \left( \limsup_{n \to \infty} y_n v_n U \right) \cap \RFM - F_{\eta}(N) \neq \emptyset.
    $$
    Since $y_n v_n U \subset Yv_n U = Y v_n$, this implies  that
    $$X_0 \cap \RFM - F_{\eta}(N) = \left( \limsup_{n \to \infty} Y v_n U \right) \cap  \RFM - F_{\eta}(N) \neq \emptyset \quad \text{for any } \eta > 0.$$
    Together with \eqref{eqn.plane1}, we must have
    $$X_0 = \RFPM,$$
    completing the proof.
\end{proof}

\subsection*{A closed $H$-orbit in $\BFM$}
To prove Theorem \ref{thm.xwh}, we first consider the case when the $U$-orbit closure contains an $H$-orbit in $\BFM \subset \RFM$.

\begin{proposition}\label{prop.yh2}
    Let $x\in\RFM$ and $X=\ov{xU}$.
    Suppose that there exists $y \in \BFM$ such that $yH \subset X$.
    Then either 
    $$X=yH \quad \text{or} \quad X=\RFPM.$$
\end{proposition}

\begin{proof}
    Note that $y \in \BFM$ implies that $yH$ is closed. Since $y \in X$, there exists a sequence $x_n \in xU \subset \RFM \cdot U$ such that $x_n \to y$ as $n \to \infty$. By Proposition  \ref{prop.intorfm}(2), after passing to a subsequence, there exists a sequence $u_n \in U$ such that $x_n u_n \in \RFM$ and $x_n u_n \to y_0$ for some $y_0 \in \BFM$. Replacing $x_n$ with $x_n u_n$, we assume that $x_n \in \RFM$ and $x_n \to y_0$.

    Then there exists a sequence $g_n \in G$ such that $g_n \to e$ and $x_n = y_0 g_n$ for all $n \ge 1$. After passing to a subsequence, $g_n \in HV$ for all $n \ge 1$ or $g_n \notin HV$ for all $n \ge 1$.
    
    If $g_n \in HV$ for all $n \ge 1$, we write $g_n = h_n v_n$ for some $h_n \in H$ and $v_n \in V$, and hence
    $$X = \ov{x_n U} = \ov{y_0 g_n U} = \ov{y_0 h_n v_n U} = \ov{y_0 h_n U}v_n \quad \text{for all } n \ge 1.$$
    Since $\ov{y_0 h_n U} \subset y_0 H$ and $y_0 H$ is compact as $y_0 \in \BFM$, it follows from Corollary \ref{cor.uinsurface} that $\ov{y_0 h_n U} = y_0 H$. Therefore,
    $$X = y_0 H v_n \quad \text{for all } n \ge 1.$$
    Since $X$ already contains a closed $H$-orbit $yH$, we must have $v_n = e$ and $$X = y_0H = yH$$ in this case.

    Now consider the case that $g_n \notin HV$ for all $n \ge 1$. Let $k_{\M} > 1$ and $R > 0$ as in Proposition \ref{prop.tt} and Proposition \ref{prop.thickness} respectively.
    Let $\rho > 0$ be such that $x_n \to y_0$ in $W_\rho$. Then by Proposition \ref{prop.thickness}(1), the set $T_n := \T_{W_{\rho + R}}(x_n)$ is $4k_{\M}$-thick. Applying Lemma \ref{lem.seqrecurrence} to $g_n \in G - HV$ and $T_n \subset  U$, we have an unbounded subset $V_0 \subset V$ such that
    $$V_0 \subset \limsup_{n \to \infty} H g_n T_n.$$
    In particular, for any $v \in V_0$, there exist sequences $h_n \in H$ and $u_n \in T_n$ such that $h_n g_n u_n \to v$ as $n \to \infty$. Since $y_0H$ is compact, $y_0 h_n^{-1} \in y_0 H$ converges to some element, say $y_1 \in y_0 H$. Then for each $n \ge 1$,
    $$
    y_0 h_n^{-1} (h_n g_n u_n) = x_n u_n \in X \cap W_{\rho + R}
    $$
    and hence taking $n \to \infty$, we have  $y_1 v \in X \cap W_{\rho + R}$.
    Therefore,
    \be \label{eqn.middlebfmcontain}
    y_0 H v = \ov{y_1 U} v = \ov{y_1 v U} \subset X
    \ee
    where the first equality is due to Corollary \ref{cor.uinsurface}.
    Moreover, since $y_1 v \in W_{\rho + R} \subset \RFM$ and $y_1 v \in y_0 H v$, we in particular have that $y_0 H v \cap \RFM \cdot U \neq \emptyset$.
    Since this holds for any $v \in V_0$ and $V_0 \subset V$ is unbounded, 
    we take any sequence $v_n \to \infty$ in $V_0$ and apply Lemma \ref{lem.yvn} to conclude that $$\limsup_{n \to \infty} (y_0 H v_n \cap \RFPM) = \RFPM.$$
    Since $y_0HV_0 \subset X$ as in \eqref{eqn.middlebfmcontain}, this implies
    $$X = \RFPM,$$
    as desired.
\end{proof}

\subsection*{A closed $H$-orbit outside $\BFM$}

We next turn to the case that $yH$ is disjoint from $\BFM$.
Unlike the case that the $H$-orbit is contained in $\BFM$, $yH$ can contain a compact $U$-orbit.
The proof of the following lemma is similar to that of Lemma \ref{lem.xt}:
\begin{lemma}\label{lem.AO}
    Let $x_n \in \RFM$ be a sequence converging to $y \in \RFM$ such that the $A$-orbit  $yA$ is compact. Let $k_{\M} > 1$ be as in Proposition \ref{prop.tt}.
    Then  for any $\eta>0$, there exists a neighborhood $O_\eta(N) \subset \Ga \ba G$ of $F_\eta(N)$ such that
    $\RFM - O_{\eta}(N)$ is compact and 
    $$
    \T_{\RFM - O_{\eta}(N)}(x_n) = \{ u \in U : x_nu \in\RFM-O_\eta(N)\}
    $$
    is $4k_{\M}$-thick for all sufficiently large $n \ge 1$.
\end{lemma}
\begin{proof}
    Let $R>0$ be the constant as in Proposition \ref{prop.thickness}.
    We choose $\rho \ge 0$ and set $\xi := e^{-2 \rho} \xi_{\M}$ such that
    $$y A \subset W_{2\rho} \subset \inte(W_{\rho}) = \inte( \RFM - \inte( F_{\xi} ))$$
    which is possible by the compactness of $yA$.

    We set $\xi_1 := e^{-2R}\xi = e^{-2(\rho  + R)} \xi_{\M}$.
    Let $\eta > 0$ be arbitrary and fix $s > 0$ such that $e^{-2s} \eta < 0.5 \xi_1$. We set $O_{\eta}(N) := \inte (F_{\xi_1}) a_s^{-1}$, which is an open neighborhood of $F_{\eta}(N)$ since
    $$
    \op{int}(F_{\xi_1})a_{s}^{-1}\supset F_{0.5\xi_1}a_{s}^{-1}\supset F_{0.5\xi_1}(N)a_{s}^{-1}=F_{0.5e^{2s}\xi_1}(N)\supset F_{\eta}(N).
    $$
    Moreover, $\RFM - O_{\eta}(N) = ( \RFM - \inte(F_{\xi_1})) a_s^{-1}$ is compact.

    To see the thickness, note that for $u \in U$,
    $$\begin{aligned}
    x_n u \in \RFM - O_{\eta}(N) & \Leftrightarrow x_n u \in \RFM - \inte(F_{\xi_1}) a_s^{-1} \\
    & \Leftrightarrow (x_n a_s) (a_s^{-1} u a_s) \in \RFM - \inte(F_{\xi_1}) \\
    & \Leftrightarrow a_s^{-1} u a_s \in \T_{W_{\rho + R}}(x_n a_s) 
    \end{aligned}
    $$
    since $W_{\rho + R} = \RFM - \inte(F_{\xi_1})$.
    Observing that $a_s^{-1} \begin{pmatrix} 1 & t \\ 0 & 1 \end{pmatrix} a_s = \begin{pmatrix} 1 & t e^{-s} \\ 0 & 1 \end{pmatrix}$ for all $t \in \R$, it suffices to show that $\T_{W_{\rho + R}}(x_n a_s)$ is $4k_{\M}$-thick.

    On the other hand, since $yA$ is contained in the open subset  $\inte(W_{\rho}) \subset \RFM$, it follows from $x_n a_s \to ya_s$ that  $x_na_s \in W_{\rho}$ for all large $n \ge 1$. Therefore, $\T_{W_{\rho + R}}(x_n a_s)$ is $4k_{\M}$-thick by Proposition \ref{prop.thickness}(1). This completes the proof.
\end{proof}

Recall that $v_t = \begin{pmatrix} 1 & it \\ 0 & 1 \end{pmatrix} \in V$ for $t \in \R$. 

\begin{lemma} \label{lem.etar}
    Let $x\in\RFPM$ and $X=\ov{xU}$. 
    Let $y \in \RFM$ be such that $yH$ is closed.
    Let  $I \subset \R$ be a compact subset.
    Suppose that for any $\eta>0$, there exist $t_\eta\in I$ and $y_\eta\in yH\cap\RFPM$ such that $y_\eta v_{t_\eta}\in X-F_\eta(N)$.
    Then there exists $T_I \in I$ such that 
    $$
    (yH\cap\RFPM)v_{T_I}\subset X.
    $$
\end{lemma}

\begin{proof}
Since $U$ and $V$ commute and $X$ is $U$-invariant, $\ov{y_{\eta}U}v_{t_{\eta}}\subset X$ for all $\eta > 0$. Hence, if  $\ov{y_{\eta_0}U}=yH\cap\RFPM$ for some $\eta_0>0$, setting $T_I:=t_{\eta_0}$ verifies the claim.

    We now assume that $\ov{y_{\eta} U} \neq yH \cap \RFPM$ for all $\eta > 0$. 
    By Corollary \ref{cor.uinsurface}, $\ov{y_{\eta} U} \neq yH \cap \RFPM$ implies that $y_{\eta} U$ is compact for all $\eta > 0$.
    Since $I$ is compact, there exists a sequence $\eta_n > 0$ such that $\eta_n \to \infty$ and $t_n := t_{\eta_n} \in I$ converges to some $T_I \in I$ as $n \to \infty$. 
    Since $y_{\eta_n} v_{t_n}\not\in F_\eta(N)$, or equivalently $y_{\eta_n} \not\in F_\eta(N)$, we have
    $$
    \limsup_{n \to \infty} y_{\eta_n} U =yH\cap \RFPM
    $$
    by Lemma \ref{lem.OS}.  Since $v_{t_n} \to v_{T_I}$ as $n \to \infty$, this implies
    $$
    \limsup_{n \to \infty} y_{\eta_n} v_{t_n} U = \limsup_{n \to \infty} y_{\eta_n} U v_{t_n} = (yH \cap \RFPM) v_{T_I}.
    $$
    Since $X$ is a closed $U$-invariant set and $y_{\eta_n} v_{t_n} \in X$ for all $n \ge 1$, it follows that $$(yH \cap \RFPM) v_{T_I} \subset X,$$
    as desired.
\end{proof}

We now classify the $U$-orbit closure containing a closed $H$-orbit outside of $\BFM$.

\begin{proposition}\label{prop.yh1}
    Let $x\in\RFM$ and $X=\ov{xU}$.
    Suppose that there exists $y \in \RFM - \BFM$ such that $yH$ is closed and $yH\cap\RFPM\subset X$.
    Then either 
    $$X=yH\cap\RFPM \quad \text{or} \quad X=\RFPM.$$
\end{proposition}
\begin{proof}
    By Proposition \ref{prop.noneltgf}, there exists a compact $A$-orbit in $yH$. Hence, by replacing $y$ with an element of $yH$, we may assume that $yA$ is compact. Note that we still have $y \in \RFM - \BFM$ after the replacement due to Lemma \ref{lem.closedH}.
    
    Since $y \in X$, there exists a sequence $x_n\in xU$ such that $x_n\to y$ as $n\to\infty$.
    Since $y\not\in\BFM$, by Proposition \ref{prop.intorfm}(1), we may further assume that $x_n\in\RFM$ by modifying $x_n\in xU$, without changing $y$.
    
    We next write $x_n=yg_n$ for some sequence $g_n \in G$ such that $g_n\to e$ as $n\to\infty$. After passing to a subsequence, $g_n \in HV$ for all $n \ge 1$ or $g_n \notin HV$ for all $n \ge 1$. Suppose first that $g_n \in HV$ for all $n \ge 1$. For each $n \ge 1$, write $g_n=h_nv_n$ for some $h_n \in H$ and $v_n \in V$. We then have 
    $$X =\ov{x_nU}=\ov{yh_nU}v_n \quad \text{for all } n \ge 1.$$
    Note that $yh_n\in yH\cap\RFPM$ and hence either $yh_nU$ is compact or $\ov{yh_nU}=yH\cap\RFPM$ by Corollary \ref{cor.uinsurface}.
    Since $yH\cap\RFPM\subset X$ by hypothesis, necessarily $v_n=e$ and $X=yH\cap\RFPM$ in this case.
    
    Now consider the case that $g_n\not\in HV$ for all $n \ge 1$.
    Let $k_{\M} > 1$ be given in Proposition \ref{prop.tt} and $k' > 1$ the constant given in Lemma \ref{lem.seqrecurrence} but associated to $k = 4k_{\M}$ instead of $k_{\M}$.

    \medskip

    We claim that for every $\eta>0$ and $r>0$,
    $$\begin{matrix}
    \exists t_{\eta,r}\in [-k'r, -r] \cup [r,k'r] \quad \text{and}  \quad y_{\eta,r}\in yH\cap\RFPM \\
    \\
    \text{such that } y_{\eta,r}v_{t_{\eta,r}}\in X-F_\eta(N).
    \end{matrix}$$
    Let us verify the claim.
    Fix arbitrary $\eta$, $r>0$ and let $O_\eta(N)$ be the open set in Lemma \ref{lem.AO} associated to $\eta>0$.
    Then $T_n := \mathsf{T}_{\RFM-O_{\eta}(N)}(x_n)$
    is $4k_{\M}$-thick for all sufficiently large $n \ge 1$ by Lemma \ref{lem.AO}.
    Next, apply Lemma \ref{lem.seqrecurrence} to obtain a $k'$-thick set $V_0\subset V$ with $k'=k'(4k_{\M})$ such that $$V_0 \subset \limsup_{n \to \infty} H g_n T_n.$$
    Since $V_0$ is $k'$-thick, we can find $t_{\eta,r}\in [-k' r, -r] \cup [r,k'r]$ such that $v_{t_{\eta,r}}\in V_0$. Hence, there exist sequences $h_n\in H$, $u_n \in T_n$ such that $h_ng_nu_{n}\to v_{t_{\eta,r}}$ as $n\to\infty$. 
    Note that
    $$
    (yh_n^{-1})(h_ng_nu_{n})=x_nu_{n}\in X\cap \RFM-O_\eta(N) \quad \text{for all } n \ge 1.   
    $$
    Since $\RFM-O_\eta(N)$ is compact, $x_nu_{n}$ is convergent after passing to a subsequence, and so is $yh_n^{-1}$, to some element $y_{\eta,r}\in yH$.
    Since $O_\eta(N)$ is a neighborhood of $F_\eta(N)$, taking the limit $n\to\infty$, we obtain $y_{\eta,r}v_{t_{\eta,r}}\in X-F_\eta(N)$.
    Since $X\subset\RFPM$, we have $y_{\eta,r}\in\RFPM$ and the claim follows.

        \medskip

    Fixing $r>0$ and by varying $\eta>0$, we apply Lemma \ref{lem.etar} to the compact set $[-k'r, -r] \cup [r, k'r]$.
    The previous claim and Lemma \ref{lem.etar} imply that there exists $T_{r}\in [-k'r, -r] \cup [r, k'r]$ such that 
    \be \label{eqn.Tr}
    (yH\cap\RFPM)v_{T_{r}}\subset X
    \ee
    Repeating this for increasing values of $r>0$, we obtain a sequence $|T_{r}| \to\infty$.

    Since $y \notin \BFM$, we in particular have $y \notin \BFM \cdot V$ by Lemma \ref{lem.closedH}. This implies that
    $
    y v_{T_r} \notin \BFM \cdot V$ for all $ r > 0$.
    Since $y \in \RFM$, we have $y v_{T_r} \in \RFPM$, and hence $y v_{T_r} \in \RFM \cdot U$
    by Lemma \ref{lem.zv}. In particular, $$y H v_{T_r} \cap \RFM \cdot U \neq \emptyset \quad \text{for all } r > 0.$$
    Since $v_{T_r} \to \infty$ as $r \to \infty$, it follows from Lemma \ref{lem.yvn} that 
    $$
    \limsup_{r \to \infty} (y H v_{T_r} \cap \RFPM) = \RFPM.
    $$
    Together with \eqref{eqn.Tr}, $X = \RFPM$. This finishes the proof.
    \end{proof}

    \subsection*{Proof of Theorem \ref{thm.xwh}} Let $x \in \RFM$ and $X = \ov{xU}$. Suppose that there exists $y \in \RFM$ such that $yH$ is closed and $yH \cap \RFPM \subset X$. If $y \in \RFM - \BFM$, then the claim follows from Proposition \ref{prop.yh1}. If $y \in \BFM$, then $X = yH$ or $X = \RFPM$ by Proposition \ref{prop.yh2}. Since $yH \subset \BFM \subset \RFM$ in this case, this completes the proof.
    \qed

%%%%%%%%%%%%%%%%%%%%%%%%%%%%%%%%%%%%%%%%%%%%
%%%%%%%%%%%%%%%%%%%%%%%%%%%%%%%%%%%%%%%%%%%%
%%%%%%%%%%%%%%%%%%%%%%%%%%%%%%%%%%%%%%%%%%%%
%%%%%%%%%%
%%%%%%%%%% X meeting N
%%%%%%%%%%
%%%%%%%%%%%%%%%%%%%%%%%%%%%%%%%%%%%%%%%%%%%%
%%%%%%%%%%%%%%%%%%%%%%%%%%%%%%%%%%%%%%%%%%%%
%%%%%%%%%%%%%%%%%%%%%%%%%%%%%%%%%%%%%%%%%%%%

\section{A $U$-orbit closure meeting a compact $N$-orbit}\label{sec.xwn}

Let $\Ga < G$ be a geometrically finite Kleinian group with a round Sierpi\'nski limit set and $\M = \Ga \ba \H^3$.
The goal of the section is to prove the following dichotomy for $U$-orbit closures meeting compact $N$-orbits.

\begin{theorem}\label{thm.clu2}
    Let $x\in\RFPM$ and $X := \ov{xU}$.
    Suppose that $X$ meets a compact $N$-orbit.
    Then one of the following holds:
    \begin{enumerate}
        \item $xN$ is compact.
        \item $X$ contains a $vAUv^{-1}$-orbit for some $v\in V$.
    \end{enumerate} 
\end{theorem}

The proof is based on unipotent blowup involving polynomials. Recall that for $z \in \C - \{0\}$ and $s \in \R$,
$$d(z) = \begin{pmatrix}
z^{-1} & 0 \\
0 & z
\end{pmatrix} \in AM \quad \mbox{and} \quad
v(s) = 
\begin{pmatrix}
  1 & is \\
  0 & 1
\end{pmatrix} \in V.$$

\begin{lemma}\label{lem.acc1}
    Let $x\in\RFPM$ and $X := \ov{xU}$.
    Suppose that there exists $y \in X$ such that $yN$ is compact.
    Then one of the following holds:
    \begin{enumerate}
        \item $xN$ is compact.
        \item there exist polynomials $\sigma, \nu \in \R[t]$ such that at least one of them is non-constant, $\sigma(0) = 1$, $\nu(0) = 0$, and
        $$
            \ov{yU}\Phi(t)\subset X \quad \text{for all } t \in \R \text{ s.t. } \sigma(t) \neq 0
        $$
        where $\Phi(t) = d(\sigma(t)) v(\nu(t))$.
    \end{enumerate}
\end{lemma}
\begin{proof}
    Since $y\in X$, there exists a sequence $x_n\in xU$ such that $x_n\to y$ as $n\to\infty$. We may write $x_n=yg_n$ for some sequence $g_n\to e$ in $G$. After passing to a subsequence, we have either $g_n \in \op{N}_G(U)$ for all $n \ge 1$ or $g_n \notin \op{N}_G(U)$ for all $n \ge 1$.

    Suppose first that $g_n \in \op{N}_G(U)$ for all $n \ge 1$. Then for each $n \ge 1$,
    $$
    xU = y g_n U = yU g_n,
    $$
    and hence there exists $u_n \in U$ such that $x = y u_n g_n$.
    Since $y u_n \in yN$ and $yN$ is compact, after passing to a subsequence, we may assume that $y u_n$ converges to some $z \in yN$. It then follows from $g_n \to e$ that $x = z$, and therefore
    $$xN = zN = yN \text{ is compact.}$$

    Now assume that $g_n\not\in\op{N}_G(U)$ for all $n \ge 1$.
    Applying Lemma \ref{lem.qr2}(1) with $S=U$,
    we obtain polynomials $\sigma, \nu \in \R[t]$ such that at least one of them is non-constant, $\sigma(0) = 1$, $\nu(0) = 0$, and
        \begin{equation}\label{eq.ugt1}
            \Phi(t) = d(\sigma(t)) v(\nu(t)) \in \limsup_{n \to \infty} U g_n U\quad \text{for all } t \in \R \text{ s.t. } \sigma(t) \neq 0.
        \end{equation}     
    Let $t\in\bb R$ be such that $\sigma(t)\neq 0$.
    By \eqref{eq.ugt1}, there exist sequences $\widehat u_n, u_n' \in U$ such that $\widehat u_n  g_n u_n' \to \Phi(t)$ as $n\to\infty$, after passing to a subsequence.
    Observe that 
    $$
    y \widehat u_n^{-1}( \widehat u_n g_n u_n' )=x_n u_n' \in X \quad \text{for all } n \ge 1.
    $$
    Since $yN$ is compact, passing to a subsequence, there exists $y_t\in \ov{yU}$ such that $y\widehat u_n^{-1}\to y_t$ as $n\to\infty$. This implies
    $$
    y_t \Phi(t) \in X.    
    $$
    Since $\ov{yU}$ is a $U$-minimal subset of a compact $N$-orbit $yN$, it follows from $\Phi(t) \in AV < \op{N}_G(U)$ that
    $$
    \ov{y U} \Phi(t) = \ov{y_t U} \Phi(t) = \ov{y_t \Phi(t) U} \subset X.
    $$
    Therefore, (2) holds.
\end{proof}

As a corollary, we obtain:

\begin{corollary} \label{cor.clu1}
    Let $x\in\RFPM$, and $X=\ov{xU}$. Suppose that $X$ meets a compact $N$-orbit.
    Then one of the following holds:
    \begin{enumerate}
        \item $xN$ is compact.
        \item for any $y \in X$, 
        $$yN \text{ is compact} \Longrightarrow yN \subset X.$$
        \item $X$ contains a $vAUv^{-1}$-orbit for some $v\in V$. 
    \end{enumerate}
   
\end{corollary}
\begin{proof}
    Assume that (1) and (2) do not hold.
    We will prove that (3) is the case.
    Since (2) does not hold, there exists $y \in X$ such that $yN$ is compact and  $yN\not\subset X$.
    We then apply Lemma \ref{lem.acc1}: Lemma \ref{lem.acc1}(1) cannot occur by our assumption, and hence there exist polynomials $\sigma, \nu \in \R[t]$ such that at least one of them is non-constant, $\sigma(0) = 1$, $\nu(0) = 0$, and
    \be \label{eqn.realpolyapply}
    \ov{yU}\Phi(t)\subset X \quad \text{for all } t \in \R \text{ s.t. } \sigma(t) \neq 0
    \ee
        where $\Phi(t) = d(\sigma(t)) v(\nu(t))$.
    
        \medskip

    We claim that $\sigma$ is non-constant.
    Suppose not and let
    $$
    X_1:=\limsup_{t\to \infty} \ov{yU}\Phi(t) \subset X. 
    $$    Since $\sigma(0) = 1$, $\Phi(t) \in V$ and hence $\ov{yU}\Phi(t) \subset yN$. 
    Since $yN$ is compact, we have $X_1\neq\emptyset$.
    Applying Corollary \ref{cor.blowup2} to $Y:=\ov{yU}$, it follows that $X_1$ is invariant under a one-parameter subgroup $L < AV$. By Lemma \ref{lem.1psg}, either $L = v A v^{-1}$ for some $v \in V$ or $L = V$. Since $X_1 \subset yN$ and $yN$ is compact, we must have $L = V$. Together with the $U$-invariance of $X_1$, we have $$X_1 = yN.$$
    This contradicts $yN \not \subset X$, and the claim follows.

        \medskip

    To finish the proof, suppose first that $y\Phi(t)\in\RFM\cdot U$ for all sufficiently large $t > 0$. Then by Proposition \ref{prop.sigma2}, $\limsup_{t \to \infty} yU \Phi(t)$ contains a $v AU v^{-1}$-orbit for some $v \in V$, and therefore (3) holds due to \eqref{eqn.realpolyapply}.

    Otherwise, $y\Phi(t)\in\BFM\cdot V$ for some $t\in\bb R$ by Lemma \ref{lem.zv}.
    Since $\Phi(t)\in AV$, we have $y=zv^{-1}$ for some $z\in\BFM$ and $v\in V$.
    Since $zH$ is a compact $H$-orbit, we have from Corollary \ref{cor.uinsurface} that
    $$
    \ov{yU} = \ov{zU}v^{-1} = zHv^{-1} = zv^{-1} (v H v^{-1}).
    $$
    Therefore,
    $$
    zv^{-1} (v AU v^{-1}) \subset zv^{-1} (v H v^{-1}) = \ov{yU} \subset X,
    $$
    and hence (3) holds, completing the proof.
\end{proof}

We prove one more lemma.

\begin{lemma}\label{lem.acc2}
    Let $x\in\RFPM$ and $X := \ov{xU}$.
    Suppose that there exists $y \in X$ such that $yN$ is compact.
    Then one of the following holds:
    \begin{enumerate}
        \item $xN$ is compact.
        \item there exists a non-constant polynomial $\sigma\in\bb C[t]$ with $\sigma(0) = 1$ so that
        for any $t\in\bb R$ satisfying $\sigma(t)\neq 0$, there exists $y_t\in yN$ such that
        $$
        y_td(\sigma(t)) \in X.
        $$
    \end{enumerate}
\end{lemma}
\begin{proof}
    Since $y\in X$, there exists a sequence $x_n\in xU$ such that $x_n\to y$ as $n\to\infty$.
    We may write $x_n=yg_n$ for some sequence $g_n\to e$ in $G$. After passing to a subsequence, we may assume that either $g_n \in \op{N}_G(N)$ for all $n \ge 1$ or $g_n \notin \op{N}_G(N)$ for all $n \ge 1$.

    If $g_n\in\op{N}_G(N)$ for all $n \ge 1$, then
    $$
    x \in x_n U = y g_n U \subset y g_n N = y N g_n \quad \text{for all } n \ge 1.
    $$
    Since $yN$ is compact and $g_n \to e$, this implies that $xN$ is compact, and (1) follows.
    
    Now assume $g_n\not\in\op{N}_G(N)$ for all $n \ge 1$.
    Applying Lemma \ref{lem.qr2}(2) with $S=N$, we obtain a non-constant polynomial $\sigma\in\bb C[t]$ with $\sigma(0) = 1$ satisfying
    \begin{equation}\label{eq.ngt1}
    d(\sigma(t))\in\limsup_{n \to \infty} Ng_n U \quad \text{for all } t \in \R \text{ s.t. } \sigma(t) \neq 0.
    \end{equation}
    Let $t\in\bb R$ be  such that $\sigma(t)\neq 0$.
    By \eqref{eq.ngt1}, there exist sequences $p_n \in N$ and $u_n \in U$ such that
    $p_n g_n u_n \to d(\sigma(t))$ as $n \to \infty$. 
    Since $yN$ is compact, after passing to a subsequence, $y p_n^{-1}$ converges to some $y_t \in yN$. We then have
    $$
    y_t d(\sigma(t)) = \lim_{n \to \infty} y p_n^{-1} (p_n g_n u_n) = \lim_{n \to \infty} x_n u_n \in X.
    $$
    Therefore, (2) holds.
\end{proof}

\begin{remark}
    We remark that proofs of Lemma \ref{lem.acc1} and Lemma \ref{lem.acc2} work for a general Kleinian group $\Ga < G$.
\end{remark}

\subsection*{Proof of Theorem \ref{thm.clu2}}
    Suppose that $xN$ is not compact. By Corollary \ref{cor.clu1},
    it suffices to consider the case Corollary \ref{cor.clu1}(2) that any compact $N$-orbit meeting $X$ is contained in $X$. Assume that we are in such a case. 
    By the hypothesis, there exists $y \in X$ such that $yN$ is compact.
    Since $xN$ is not compact, it follows from Lemma \ref{lem.acc2} that we can find a non-constant polynomial $\sigma\in\bb C[t]$ such that for all $t\in\bb R$ satisfying $\sigma(t)\neq 0$, we have for some $y_t \in yN$ that 
    $$
    y_td(\sigma(t))\in X.
    $$
    Note that $d(\sigma(t))\in\op{N}_G(N)$ and hence
    $$
    y_td(\sigma(t))N=y_tNd(\sigma(t))=yNd(\sigma(t))
    $$
    is a compact $N$-orbit meeting $X$. Therefore, it follows from the hypothesis that
    $$
    yNd(\sigma(t))\subset X.
    $$
    Since $\sigma \in \C[t]$ is non-constant, we have $|\sigma(t)|\to\infty$ as $t\to\infty$ in $\bb R$.
    Hence $X=\RFPM$ by Lemma \ref{lem.eqd}. In particular, $X$ contains a $vAUv^{-1}$-orbit for some $v\in V$. This finishes the proof.
\qed

%%%%%%%%%%%%%%%%%%%%%%%%%%%%%%%%%%%%%%%%%%%%
%%%%%%%%%%%%%%%%%%%%%%%%%%%%%%%%%%%%%%%%%%%%
%%%%%%%%%%%%%%%%%%%%%%%%%%%%%%%%%%%%%%%%%%%%
%%%%%%%%%%
%%%%%%%%%% X without closed U
%%%%%%%%%%
%%%%%%%%%%%%%%%%%%%%%%%%%%%%%%%%%%%%%%%%%%%%
%%%%%%%%%%%%%%%%%%%%%%%%%%%%%%%%%%%%%%%%%%%%
%%%%%%%%%%%%%%%%%%%%%%%%%%%%%%%%%%%%%%%%%%%%

\section{A $U$-orbit closure without any compact $U$-orbit}\label{sec.xwou}

Let $\Ga < G$ be a geometrically finite Kleinian group with a round Sierpi\'nski limit set and $\M = \Ga \ba \H^3$.
In this section, we consider an orbit closure $\ov{xU}$ without any compact $U$-orbit. The following is the main theorem of this section:

\begin{theorem}\label{thm.xu3}
    Let $x \in \RFM$ and $X := \overline{xU}$. Suppose that $X$ does not contain any compact $U$-orbit.
    Then one of the following holds:
    \begin{enumerate}
      \item there exists a compact $N$-orbit in $X$.
      \item there exists a $vAUv^{-1}$-orbit in $X$ for some $v \in V$.
    \end{enumerate}
  \end{theorem}

To prove Theorem \ref{thm.xu3}, we recall the notion of relatively minimal sets, introduced in \cite{McMullen2017geodesic}:

\begin{definition}\label{def.min}
    Let $W \subset \Ga \ba G$.
  A closed subset $Y \subset \Gamma \backslash G$ is called \emph{$U$-minimal relative to $W$} if $Y \cap W \neq \emptyset$ and $\overline{yU} = Y$ for all $y \in Y \cap W$.
\end{definition}

Note that a relatively $U$-minimal set is $U$-invariant.
If $W$ is compact, then any closed $U$-invariant set $Y \subset \Gamma \backslash G$ such that $Y\cap W\neq\emptyset$ contains a $U$-minimal set relative to $W$.
This usually follows from Zorn's lemma.

Given two closed subsets in $\Ga \ba G$, we collect elements of $G$ that deliver one subset to the other.

\begin{definition}\label{def.s}
For any closed subsets $Y_1$, $Y_2\subset\Ga\ba G$, We define
$$
D(Y_1,Y_2):=\{g\in G: Y_1g\cap Y_2\neq\emptyset\}.
$$    
\end{definition}

Note that for $x \in \Ga \ba G$ and a closed subset $Y \subset \Ga \ba G$,
$$ \T_Y(x) = D(\{x \}, Y) \cap U.$$

\begin{lemma}\label{lem.sy1}
    Let $Y_1, Y_2 \subset \Ga \ba G$ be closed subsets and $S_1, S_2 < G$.
\begin{enumerate}
    \item If one of $Y_i$ $(i=1,2)$ is compact, then $D(Y_1, Y_2)$ is closed.
    \item If $Y_iS_i=Y_i$ for $i = 1, 2$, then $D(Y_1, Y_2)=S_1 D(Y_1, Y_2)S_2$.
\end{enumerate}
\end{lemma}
The proof of Lemma \ref{lem.sy1} is rather straightforward and will be omitted.
\begin{lemma}\label{lem.sy2}
    Let $Y_1$, $Y_2\subset\Ga\ba G$ be closed $U$-invariant sets and $W \subset \Ga \ba G$ a compact subset.
    Suppose that $Y_1$ is $U$-minimal relative to $W$.
    Then
    $$
    D(Y_1\cap W,Y_2)\cap\op{N}_G(U)=\{g\in\op{N}_G(U):Y_1g\subset Y_2\}.
    $$
    In particular, $D(Y_1\cap W,Y_1)\cap\op{N}_G(U)$ is a closed subsemigroup of $\op{N}_G(U)$.
\end{lemma}
\begin{proof}
    The hypothesis that $Y_1$ is $U$-minimal relative to $W$ implies that $Y_1 \supset Y_1 \cap W \neq \emptyset$, and hence $\{g\in\op{N}_G(U):Y_1g\subset Y_2\} \subset D(Y_1\cap W,Y_2)\cap\op{N}_G(U)$.

    Conversely, let $g\in D(Y_1\cap W,Y_2)\cap\op{N}_G(U)$.
    Then $y_1g=y_2$ for some $y_1\in Y_1\cap W$ and $y_2\in Y_2$.
    Since $Y_1$ is $U$-minimal relative to $W$ and $g\in\op{N}_G(U)$, we have
    $$
    Y_1g=\ov{y_1U}g=\ov{y_1gU}=\ov{y_2U}\subset Y_2.
    $$
    Hence, the reverse inclusion follows. The last assertion is straightforward.
\end{proof}

We will use the following lemma:

\begin{lemma}\cite[Lemma 8.2]{Benoist2022geodesic}\label{lem.yuy}
    Let $Y \subset \Ga \ba G$ be a $U$-minimal set relative to a compact subset $W \subset \Ga \ba G$ and $y\in Y\cap W$. 
    If $\T_{Y \cap W}(y)$ is unbounded, then there exists a sequence $u_n\to\infty$ in $U$  such that $yu_n\to y$ as $n\to\infty$.
\end{lemma}

An analogue of the following proposition was proved in \cite[Theorem 9.4]{McMullen2017geodesic} when $\Ga$ is further assumed to be convex cocompact:

\begin{lem} \label{lem.newdynamics}
    Let $Y \subset \Ga \ba G$ be a $U$-minimal set relative to a compact subset $W\subset \Ga \ba G$. Suppose that $Y$ is not a compact $U$-orbit, and that 
    $\T_W(y)$ is $k$-thick at $\infty$ for some $y \in Y \cap W$ and $k > 1$.
    Then there exists a one-parameter subsemigroup $L_+<AV$ such that 
    $$
    YL_+\subset Y.
    $$
\end{lem}
\begin{proof}
Since $Y \cap W$ is compact, $D(Y \cap W, Y)$ is closed by Lemma \ref{lem.sy1}(1).
We claim that there exists a non-trivial element in $D(Y \cap W, Y) \cap AV$ arbitrarily close to $e$.
Since $AV < \op{N}_G(U)$ is closed and 
$$D(Y \cap W, Y) \cap \op{N}_G(U) = \{g \in \op{N}_G(U) : Y g \subset Y\} < \op{N}_G(U)$$
is a closed subsemigroup by Lemma \ref{lem.sy2}, the lemma follows from the claim.

    \medskip

Let $y\in Y \cap W$ be such that  $\mathsf{T}_W(y)$ is $k$-thick at $\infty$. Since $yU \subset Y$, $\T_{Y \cap W}(y) = \T_W(y)$ is unbounded.
By Lemma \ref{lem.yuy}, there exists a sequence $u_n\to\infty$ in $U$ such that $yu_n\to y$ as $n\to\infty$.
We can write $yu_n=yg_n$ for some sequence $g_n\to e$ in $G$ as $n\to\infty$. We in particular have 
\be \label{eqn.gisdelivery}
g_n \in D(Y \cap W, Y).
\ee

We first observe that $g_n \notin U$ for all but finitely many $n \ge 1$. Since $\ov{yU} = Y$ and $Y$ is not a compact $U$-orbit, $yU$ is not compact. Hence, if $g_n \in U$, then it follows from $y u_n g_n^{-1} = y$ that $u_n = g_n$. Since $u_n \to \infty$ and $g_n \to e$ as $n \to \infty$, this is possible only for finitely many $n \ge 1$.

We now construct a non-trivial sequence $\ell_n \to e$ in $D(Y \cap W, Y) \cap AV$, which completes the proof as mentioned above. After passing to a subsequence, either  $g_n \in AN$ for all $n \ge 1$ or $g_n \notin AN$ for all $n \ge 1$.
First, consider the case that $g_n \in AN$ for all $n \ge 1$. Since $N = VU$, there exists a sequence $\widehat u_n \to e$ in $U$ such that $g_n \widehat u_n \in AV$ for all $n \ge 1$, and $g_n \widehat u_n \to e$ in particular. By the above observation, $g_n \widehat u_n \neq e$ for all large enough $n \ge 1$. By \eqref{eqn.gisdelivery} and Lemma \ref{lem.sy1}(2), we have $g_n \widehat u_n \in D(Y \cap W, Y)$ for all $n \ge 1$. Therefore, we take $\ell_n = g_n \widehat u_n$ in this case.

Now assume that $g_n \notin AN$ for all $ n \ge 1$. Let
$$
T :=  \T_W(y)^{-1}.
$$
For any $u \in T$ and $n \in \N$, we have 
$$ yu^{-1} \in Y \cap W \quad \text{and} \quad y u^{-1} (u g_n) = y g_n  \in Y.
$$
This implies $T g_n \subset D(Y \cap W, Y)$ for all $n \ge 1$. By Lemma \ref{lem.sy1}(2),
$$T g_n U \subset D(Y \cap W, Y) \quad \text{for all } n \ge 1.$$
Since $\T_W(y)$ is $k$-thick at $\infty$, $T \subset U$ is so. 
By Lemma \ref{lem.recurrence}, $\limsup_{n \to \infty} T g_n U$ contains a sequence $\ell_n \to e$ in $AV - \{e\}$ as $n\to\infty$.
Since $D(Y \cap W, Y)$ is closed, $\limsup_{n \to \infty} T g_n U \subset D(Y \cap W, Y)$. Therefore, $\ell_n \in D(Y \cap W, Y) \cap AV$ is the desired sequence, finishing the proof.
\end{proof}

\begin{remark}
    We remark that Lemma \ref{lem.yuy} and Lemma \ref{lem.newdynamics} hold for more general geometrically finite, acylindrical hyperbolic 3-manifolds, as the original statement of Lemma \ref{lem.yuy} in \cite{Benoist2022geodesic} was proved in such a setting. The proof of Lemma \ref{lem.newdynamics} works verbatim as long as Lemma \ref{lem.yuy} holds.
\end{remark}

\subsection*{Proof of Theorem \ref{thm.xu3}}
Let $k_{\M} > 1$ and $R, \xi_{\M} > 0$ be constants as in Proposition \ref{prop.tt}, Proposition \ref{prop.thickness}, and \eqref{eqn.defxi0} respectively. Since $x \in \RFM$, there exists $\rho > 0$ such that $x \in W_{\rho + R}$. We simply write $W := W_{\rho + R}$ which is compact. Since $X = \ov{xU}$ is $U$-invariant and $X \cap W \neq \emptyset$, there exists a $U$-minimal set $Y \subset  X$ relative to $W$.

Since $Y$ cannot be a compact $U$-orbit and $\T_W(y)$ is $4k_{\M}$-thick at $\infty$ for any $y \in Y \cap W$ by Proposition \ref{prop.thickness}(2), it follows from Lemma \ref{lem.newdynamics} that there exists a one-parameter subsemigroup $L_+<AV$ such that $YL_+\subset Y$. 

    \medskip

  We claim that one of the following holds:
  \begin{enumerate}
      \item[(a)] $Y$ contains a compact $N$-orbit. 
      \item[(b)]
      there exist $y_0 \in Y$ and $v_0 \in V$ such that $y_0 v_0  \in \BFM$ and $$ y_0v_0Hv_0^{-1} \subset Y.$$
      \item[(c)] there exists a sequence $\ell_n\to\infty$ in $L_+$ such that $$\limsup_{n \to \infty} Y\ell_n \neq \emptyset.$$
  \end{enumerate}
    Suppose first that there exists $z \in Y$ such that $z N$ is compact. Since $z U \subset Y \subset X$ and $X$ does not contain any compact $U$-orbit, $z U$ is not compact. This implies
    $$z N = \ov{ z U} \subset Y,$$
    from which (a) follows.

Now assume that
\be \label{eqn.hypothesisonY}
z N \text{ is not compact for all } z \in Y.
\ee
Fix $y \in Y \cap W$ and let $g \in G$ be such that $y = [g]$. Let $L < AV$ be the one-parameter subgroup containing $L_+$.
By Lemma \ref{lem.1psg}, $L=vAv^{-1}$ for some $v\in V$ or $L=V$. 
Suppose first that $L = v A v^{-1}$ for some $v \in V$. There are two possible cases:
  \begin{itemize}
      \item Suppose that $yvU \cap \RFM = \emptyset$. By Lemma \ref{lem.zv}, we have $yv \in \BFM \cdot V$, and hence $yv_0 \in \BFM$ for some $v_0 \in V$. Recalling that $\BFM \subset \RFM$ is a union of finitely many compact $H$-orbits,
      \be \label{eqn.nowinbfm}
 y v_0 H = \ov{y v_0 U} = \ov{y U} v_0 = Y v_0
\ee
since $y v_0 H$ is $U$-minimal by Corollary \ref{cor.uinsurface}. Therefore, (b) follows in this case.

      \item Suppose that $yvU \cap \RFM \neq \emptyset$. Then we have either $(gvU)^- \subset \La$ or $(gvU)^-$ meets both $\La$ and $\widehat \C - \La$. In any case, there exists $u \in U$ such that $(guv)^- = (gvu)^- \in \La$ is conical. Hence, we can find a sequence $t_n \to \infty$ such that
      \be \label{eqn.backwardcon}
      \text{the sequence } yuva_{-t_n} \text{ converges.}
      \ee
      By Lemma \ref{lem.ferte}, $(guv)^+ = g^+$ is conical as well, since $yN$ is not compact. We then have that for some sequence $s_n \to \infty$,
      \be \label{eqn.forwardcon}
      \text{the sequence } yuva_{s_n} \text{ converges.}
      \ee
      Writing $A^+ := \{a_t \in A : t \ge 0\}$, there are two subcases:
      \begin{itemize}
          \item if $L_+ = v A^+ v^{-1}$, set $\ell_n = v a_{s_n}v^{-1} \in L_+$. Then by \eqref{eqn.forwardcon}, the sequence $$yu \ell_n = yu(va_{s_n}v^{-1})$$converges.
          \item if $L_+ = v (A^+)^{-1} v^{-1}$, set $\ell_n = v a_{-t_n}v^{-1} \in L_+$. Then by \eqref{eqn.backwardcon} the sequence $$yu \ell_n = yu(va_{-t_n}v^{-1})$$ converges.
      \end{itemize}
      Since $yu \in Y$, we have $\limsup_{n \to \infty} Y \ell_n \neq \emptyset$. Since $\ell_n \in L_+$ diverges, this shows (c).
  \end{itemize}
  
  Next, suppose  $L=V$. 
  There are two cases:
  \begin{itemize}
      \item if $Y \cap \BFM \cdot V \neq \emptyset$, then $y_0 v_0 \in \BFM$ for some $y_0 \in Y$ and $v_0 \in V$. Recalling that $\BFM \subset \RFM$ is a union of finitely many compact $H$-orbits, 
      $$
      y_0 v_0 H = \ov{y_0 v_0 U} = \ov{y_0 U} v_0 \subset Y v_0
      $$
      since $y_0 v_0 H$ is $U$-minimal by Corollary \ref{cor.uinsurface} and $Y = \ov{yU}$ is $U$-invariant. Therefore, (b) follows in this case.
      \item if $Y \cap \BFM \cdot V = \emptyset$, then $Yv \subset \RFM \cdot U$ for all $v \in V$ by Lemma \ref{lem.zv}. Choose any sequence $\ell_n \to \infty$ in $L_+ < V$. Then by the $U$-invariance of $Y$, there exists a sequence $y_n \in Y$  such that $y_n \ell_n \in \RFM$ for all $n \ge 1$.
      By the hypothesis \eqref{eqn.hypothesisonY}, $y_n \ell_n N = y_n N$ is not compact for all $n \ge 1$.
    We then have that
      $$
      y_n \ell_n \in \RFM - F_{e^{-2R}\xi_{\M}}(N) \quad \text{for all } n \ge 1.
      $$
      By Proposition \ref{prop.thickness}(3), 
      $$
      y_n \ell_n U \cap W_R \neq \emptyset \quad \text{for all } n \ge 1
      $$
      where $W_R = \RFM - \inte(F_{e^{-2R}\xi_{\M}})$. Since $y_n \ell_n U \subset Y \ell_n$ and $W_R$ is compact, it follows that $\limsup_{n \to \infty} Y \ell_n \neq \emptyset$, and therefore (c) holds.
  \end{itemize}

  \medskip

  To finish the proof, it suffices to consider the case (c) of the claim that for some sequence $\ell_n \to \infty$ in $L_+$, we have $\limsup_{n \to \infty} Y \ell_n \neq \emptyset$. Let $y_0 \in \limsup_{n \to \infty} Y \ell_n$ and take a sequence $y_n \in Y$ such that $y_n \ell_n \to y_0$ as $n \to \infty$. Then for any $\ell \in L$, we have $\ell_n \ell \in L_+$ for all large enough $n \ge 1$, and hence
  $$y_n \ell_n \ell \in Y L_+ \subset Y.$$
  Taking the limit $n \to \infty$, this implies $y_0 \ell \in Y$. Since $\ell \in L$ is arbitrary, we have $y_0 L \subset Y$, and hence
  $$
  y_0 LU \subset Y.
  $$
  Again, $L = v A v^{-1}$ for some $v \in V$ or $L = V$ by Lemma \ref{lem.1psg}.
  If $L = v A v^{-1}$ for some $v \in V$, then $y_0 v AU v^{-1} \subset Y$. If $L = V$, then $y_0 N \subset Y$. By Lemma \ref{lem.ferte}, we have either $\ov{y_0 N} = \RFPM$ or $y_0 N$ is compact. Since $\ov{y_0 N} \subset X$, this completes the proof.
\qed

%%%%%%%%%%%%%%%%%%%%%%%%%%%%%%%%%%%%%%%%%%%%
%%%%%%%%%%%%%%%%%%%%%%%%%%%%%%%%%%%%%%%%%%%%
%%%%%%%%%%%%%%%%%%%%%%%%%%%%%%%%%%%%%%%%%%%%
%%%%%%%%%%
%%%%%%%%%% Classification of horocycles
%%%%%%%%%%
%%%%%%%%%%%%%%%%%%%%%%%%%%%%%%%%%%%%%%%%%%%%
%%%%%%%%%%%%%%%%%%%%%%%%%%%%%%%%%%%%%%%%%%%%
%%%%%%%%%%%%%%%%%%%%%%%%%%%%%%%%%%%%%%%%%%%%

\section{The classification}\label{sec.fin}

In this last section, we prove our classification of $U$-orbit closures.
We restate Theorem \ref{thm.main} below:

\begin{theorem} \label{thm.mainrestate}
    Let $\M = \Ga \ba \H^3$ be a geometrically finite hyperbolic 3-manifold with a round Sierpi\'nski limit set. Then for any $x \in \FM$, one of the following holds:
    \begin{enumerate}
        \item $xU$ is closed.
        \item $\ov{xU} = xN$ which is compact.
        \item $\ov{xU} = xvHv^{-1} \cap \RFPM$ for some $v \in V$.
        \item $\ov{xU} = \RFPM$.
    \end{enumerate}
\end{theorem}

In the rest of the section, let $\M$ and $\Ga$ be as in Theorem \ref{thm.mainrestate}. 
As observed in Theorem \ref{thm.xu3} and Theorem \ref{thm.clu2}, it sometimes happen that a $U$-orbit closure contains a $vAUv^{-1}$-orbit for some $v \in V$. In this regard, we investigate $AU$-orbit closures further.

\subsection*{$AU$-orbit closures} \label{sec.au}

Using the classification of $H$-orbit closures (Theorem \ref{thm.plane}), we obtain the following:

\begin{lemma} \label{lem.ua3}
    For any $x \in \RFPM$, there exist $y \in \RFPM$ and $v \in V$ such that $yv \in \RFM$, $yvH$ is closed, and $$ yvHv^{-1} \cap \RFPM \subset \overline{xAU}.$$
\end{lemma}

\begin{proof}
    Suppose first that $x \in \RFM \cdot U$.  We take $u \in U$ so that $xu \in \RFM$. 
Recall $K = \op{PSU}(2)$ from \eqref{eqn.notation} and let $K_H = H \cap K$.
    Since $H=AUK_H$ and $K_H$ is compact, we have
    $$\overline{xAU}K_H = \overline{xH} = \overline{xuH}.$$
    By Theorem \ref{thm.plane}, either $xuH$ is closed or $\overline{xuH} = \RFPM \cdot H$.
    \begin{itemize}
        \item If $xuH$ is closed, then $$\overline{xAU} \cap (xuH \cap \RFPM) \neq \emptyset$$
        since $K_H < H$.
        By the $AU$-minimality of $xuH \cap \RFPM$ (Corollary \ref{cor.auminimal}), this proves the claim with $y = xu \in \RFM$ and $v = e$.

        \item Otherwise, we have
        $$\overline{xAU} K_H= \RFPM \cdot H.$$
        Let $y \in \BFM$. Then $yH \subset \RFM$ is closed. Hence $$\overline{xAU} \cap yH \neq \emptyset,$$
        from which we can deduce the claim as above.
    \end{itemize}

    Now suppose that $x \notin \RFM \cdot U$. By Lemma \ref{lem.zv}, $x \in \BFM \cdot V$, and hence $xv \in \BFM$ for some $v \in V$. We then have $xv H \subset \RFM$ and $xvH$ is compact. This implies that $xvH$ is $U$-minimal by Corollary \ref{cor.uinsurface}, and hence
    $$
    \ov{xAU} \supset \ov{xU} = \ov{xvU}v^{-1} = xvHv^{-1}.
    $$
    Setting $y = x$, this finishes the proof.
\end{proof}

\begin{corollary} \label{cor.ua2}
    For any $x \in \RFPM$ and $v_0 \in V$, there exist $y \in \RFPM$ and $v \in V$ such that $yv \in \RFM$, $yvH$ is closed, and $$
    yvHv^{-1} \cap \RFPM \subset \overline{xv_0AUv_0^{-1}}.
    $$
\end{corollary}

\begin{proof}
    By Lemma \ref{lem.ua3}, there exists $y' \in \RFPM$ and $v' \in V$ such that $y'v' \in \RFM$, $y'v'H$ is closed, and 
    $$
    y'v'Hv'^{-1} \cap \RFPM \subset \overline{xv_0 AU}.
    $$
    This implies that $$
    y'v_0^{-1} (v_0v'Hv'^{-1}v_0^{-1}) \cap \RFPM \subset 
    \overline{xv_0AUv_0^{-1}}.
    $$
    We set $y = y' v_0^{-1} \in \RFPM$ and $v = v_0v' \in V$. Then $yv = y'v' \in \RFM$ and $yvH = y'v'H$ is closed. Moreover, the above inclusion is rewritten as follows:
    $$yv H v^{-1} \cap \RFPM \subset \overline{x v_0 AU v_0^{-1}}.$$
    This completes the proof.
\end{proof}

We now show the intermediate classification of $U$-orbit closures based on results from Section \ref{sec.xwn} and Section \ref{sec.xwou} as follows:
\begin{prop}\label{prop.main1}
Let $x \in \RFPM$. Then one of the following holds:
    \begin{enumerate}
        \item $xU$ is closed.
        \item $\ov{xU}=xN$ which is compact.
        \item there exist $y \in \RFPM$ and $v \in V$ such that $yv \in \RFM$, $yvH$ is closed, and $$yvHv^{-1}\cap\RFPM \subset \ov{xU}.$$
    \end{enumerate}
\end{prop}
\begin{proof}

    We first note that $xN$ is compact or $\ov{xN} = \RFPM$ by Lemma \ref{lem.ferte}. This implies that if $\ov{xU} = xN$, and hence $xN$ is closed, then $xN$ is compact. Therefore, it suffices to show that if $\ov{xU}$ is neither $xU$ nor $xN$, then (3) holds.
    
    Suppose that $\ov{xU}$ is neither $xU$ nor $xN$. If $x \in \BFM \cdot V$, then $xv \in \BFM \subset \RFM$ for some $v \in V$, and hence $xvH$ is a compact $U$-minimal set by Corollary \ref{cor.uinsurface}. Therefore, $$xvH = \ov{xvU} = \ov{xU}v,$$
    from which (3) follows.

    Now assume that $x \notin \BFM \cdot V$. By Lemma \ref{lem.zv}, $x \in \RFM \cdot U$ and hence we may assume that $x \in \RFM$ by replacing $x$ with an element of $xU$.
    We claim that $\ov{xU}$ contains a $v_0AUv_0^{-1}$-orbit for some $v_0 \in V$. Once we show the claim, we apply Corollary \ref{cor.ua2} and (3) follows, finishing the proof.

    To see the claim, first consider the case that $\ov{xU}$ does not contain any compact $U$-orbit. By Theorem \ref{thm.xu3},
    $\ov{xU}$ contains a $v_0AUv_0^{-1}$-orbit for some $v_0 \in V$ or a compact $N$-orbit. In the former case, we are done. If $\ov{xU}$ contains a compact $N$-orbit, it follows from Theorem \ref{thm.clu2} that $\ov{xU}$ contains a $v_0AUv_0^{-1}$-orbit for some $v_0 \in V$ or $xN$ is compact. On the other hand, since $\ov{xU}$ is neither $xU$ nor $xN$, $xN$ cannot be compact. Indeed, if $xN$ were compact, then $xU$ is either compact or dense in $xN$, which is not the case here. Therefore, the claim follows.

    Now suppose that $\ov{xU}$ contains a compact $U$-orbit, say $yU$. If $g \in G$ is such that $y = [g]$, then $g^+$ is a parabolic limit point, and hence $yN$ is compact by Lemma \ref{lem.ferte}. Hence $\ov{xU}$ meets a compact $N$-orbit and the claim follows from Theorem \ref{thm.clu2} as above. This finishes the proof.
\end{proof}

Combined with the results from Section \ref{sec.xwh}, we now complete the classification.

\subsection*{Proof of Theorem \ref{thm.mainrestate}}
    Let $x \in \FM$. If $x \notin \RFPM$, then it is easy to see that $xU$ is closed, noting that the $\Ga$-action on $\H^3 \cup \widehat \C$ is a non-elementary convergence action with the limit set $\La$.

    Hence, we assume that $x \in \RFPM$. 
    Suppose first that $x \in \BFM \cdot V$. Then $xv \in \BFM$ for some $v \in V$, and hence
    $$
    \ov{xU} = \ov{xvU}v^{-1} = xvHv^{-1}
    $$
    since $xvH \subset \BFM \subset \RFM$ is a compact $U$-minimal set by Corollary \ref{cor.uinsurface}. Therefore, (3) follows in this case.

    It remains to consider the case $x \notin \BFM \cdot V$. By Lemma \ref{lem.zv}, we have $x \in \RFM \cdot U$. 
    By Proposition \ref{prop.main1}, it suffices to consider the case that there exist $y \in \RFPM$ and $v \in V$ such that $yv \in \RFM$, $yvH$ is closed, and
    $$yvHv^{-1} \cap \RFPM \subset \ov{xU}.$$
    Since $x \notin \BFM \cdot V$, $xv \notin \BFM \cdot V$ and hence $xv \in \RFM \cdot U$ by Lemma \ref{lem.zv}. Replacing $x$ with an element of $xU$, we may assume that $xv \in \RFM$.
    
    If $yv \in \BFM$, then
    $$
    yvH = yvH \cap \RFPM \subset \ov{xvU}.
    $$
    By Proposition \ref{prop.yh2}, this implies that either
    $$
    \ov{xU}v = \ov{xvU} = yvH \quad \text{or} \quad \ov{xU}v = \RFPM.
    $$
    In any case, (3) or (4) follows.

    If $yv \notin \BFM$, it then follows from Proposition  \ref{prop.yh1} that either 
    $$
    \ov{xU}v = yvH \cap \RFPM \quad \text{or} \quad \ov{xU}v = \RFPM.
    $$
    Again, (3) or (4) follows in either case. This completes the proof.
\qed

%%%%%%%%%%%%%%%%%%%%%%%%%%%%%%%%%%%%%%%%%%%%%%%%%%%%%%%%%%%%%%%%%%%%%%%%%%%%%%%
%
%	References
%
%%%%%%%%%%%%%%%%%%%%%%%%%%%%%%%%%%%%%%%%%%%%%%%%%%%%%%%%%%%%%%%%%%%%%%%%%%%%%%%

\medskip
\bibliographystyle{plain} 
%\bibliography{Horocycles}

\end{document}